\documentclass[11pt]{amsart}

\usepackage{amsmath,amssymb,amsthm,mathtools}
\usepackage{mathrsfs}
\usepackage{geometry}
\usepackage[hidelinks]{hyperref}

\numberwithin{equation}{section}
\newtheorem{theorem}{Theorem}[section]
\newtheorem{proposition}[theorem]{Proposition}
\newtheorem{lemma}[theorem]{Lemma}
\newtheorem{corollary}[theorem]{Corollary}
\theoremstyle{definition}
\newtheorem{definition}[theorem]{Definition}
\newtheorem{example}[theorem]{Example}
\newtheorem{remark}[theorem]{Remark}

\DeclareMathOperator{\ad}{ad}
\newcommand{\g}{\mathfrak{g}}
\newcommand{\ee}{\mathfrak{e}}

\newcommand{\Lint}{L\!\int}
\newcommand{\Var}{\operatorname{Var}}
\newcommand{\Lamn}[2]{\Lambda_{#1}^{(#2)}}

\title{The Asymptotic Development of Paths on Nilmanifolds}
\author[Mac Lean]{Mark T. Mac Lean}
\address{Department of Mathematics, The University of British Columbia, Vancouver, BC}
\email{maclean@math.ubc.ca}


\begin{document}

\begin{abstract}
We prove a new Carnot-scale $n$-step nilpotent $\pi_1$-de Rham theorem: along every convergent scale sequence, the asymptotic homotopy class of a long trajectory is identified with the nilpotent development of its macroscopic horizontal path. The key analytic mechanism is an asymptotic-development theorem showing that uniform convergence of rescaled horizontal paths, together with a uniform bound on variation, forces uniform, layer-by-layer convergence of their full Carnot-rescaled nilpotent developments. The resulting theory recovers Schwartzman's theory of asymptotic cycles in step one and the $2$-step asymptotic homotopy theory of Benardete and Mitchell in step two, while extending the correspondence to graded nilpotent groups of arbitrary step.
\end{abstract}

\maketitle

\section{Introduction}
\label{sec:introduction}

What nonabelian geometric and topological information survives when a long trajectory is viewed at macroscopic scale? This question lies at a meeting point between dynamics, which studies the statistical behaviour of trajectories, and large-scale geometry and topology, which study the structure carried by those trajectories at large scale.

Manifolds with nilpotent fundamental group provide a natural setting for this question. In the abelian case, the large-scale winding of a trajectory is described by its homology class and, asymptotically, by its average displacement. In the nilpotent setting, however, trajectories can carry additional information in the successive commutator layers of the fundamental group, which emerge at distinct Carnot scales.

A nilpotent $\pi_1$-de Rham correspondence provides a language for this nonabelian information. Representations of $\pi_1(M)$ in a connected, simply connected nilpotent Lie group correspond to flat Lie-algebra-valued connection forms, and the Lie integral of such a form around a based loop is the logarithm of its holonomy. Because the exponential map is a global diffeomorphism in the nilpotent setting, passing to the logarithm loses no information. Rather, it places the holonomy in the Lie algebra, where its successive commutator layers can be isolated and studied under Carnot rescaling.

The goal of this paper is to pass this nilpotent $\pi_1$-de Rham correspondence to the macroscopic limit. More precisely, we ask whether the Carnot-rescaled homotopy class of a long trajectory is determined by its limiting horizontal motion, and what higher-order information survives in that limit.

The principal topological result is Theorem~\ref{thm:carnot-pi1-derham}, which gives a new Carnot-scale $n$-step nilpotent $\pi_1$-de Rham theorem identifying the asymptotic homotopy class of a long trajectory with the nilpotent development of its macroscopic horizontal path. The underlying perspective is supplied by Theorem~\ref{thm:macroscopic-path}: the asymptotic object is obtained by first taking the macroscopic horizontal path and then developing it nilpotently. Thus, the surviving nonabelian information is determined not merely by average displacement, but by the ordered geometry of the limiting horizontal path.

At the abelian level, classical de Rham theory detects the fundamental group only through its homology, and its asymptotic counterpart is Schwartzman’s theory of asymptotic cycles \cite{Schw57}. Benardete and Mitchell passed beyond homology by constructing an asymptotic homotopy theory in the $2$-step nilpotent setting \cite{BM93}. The present paper carries this progression through arbitrary nilpotent step. Although the $2$-step theory of Benardete and Mitchell is its natural precursor, to the author’s knowledge no counterpart of Theorem~\ref{thm:carnot-pi1-derham} has previously been available for nilpotent groups of arbitrary step.

The passage from $2$-step to general nilpotent Lie algebras is not formal. In the $2$-step case, the Baker--Campbell--Hausdorff (BCH) formula ends with a central commutator term, and the first nonabelian contribution is controlled by a single area-type correction. Beyond $2$-step, the logarithm separates into interacting homogeneous layers. Displacement, signed area, area moments, and higher nested-bracket terms must be controlled simultaneously. The central construction of this paper is the nilpotent development of paths. Through the Chacon--Fomenko Lie integral, a horizontal path in the first layer determines a path in the full nilpotent Lie algebra whose successive components retain this ordered geometric information \cite{CF91a}.

The analytic engine for this result is Theorem~\ref{thm:macroscopic-path}. It proves that convergence of rescaled horizontal paths, under a uniform bounded-variation hypothesis, forces layer-by-layer convergence of their full nilpotent developments. Theorem~\ref{thm:nilpotent-pi1-derham} establishes the finite-time $n$-step nilpotent $\pi_1$-de Rham correspondence. Theorem~\ref{thm:carnot-pi1-derham} combines these two results to pass the finite-time correspondence to the asymptotic cone.

More precisely, Theorem~\ref{thm:carnot-pi1-derham} gives the sequential Carnot-scale form of the correspondence. A long open trajectory is completed by a closing path whose developed logarithm remains bounded. If the rescaled horizontal paths converge along a sequence, with uniformly bounded variation, then the Carnot-rescaled logarithms of the resulting homotopy classes converge to the nilpotent development of the limiting macroscopic path. The limit is independent of the bounded closings. Thus, the limiting point in the asymptotic cone is not determined merely by average displacement: it retains the ordered higher-layer information generated by the limiting path.

The theory contains the preceding asymptotic theories as its first two stages. In step one, the fundamental group is seen only through its abelianization, and the construction reduces to the homological setting of Schwartzman’s asymptotic cycles. In step two, the first nonabelian layer recovers the Benardete--Mitchell correction. In higher step, genuinely new ordered interactions appear. The free $3$-step calculation shows explicitly why the $2$-step correction cannot simply be iterated.

The distinction between full-sequence and subsequential macroscopic limits is essential. Full-sequence limits are necessarily straight, so their higher Carnot-diagonal components vanish. Subsequence limits are much less rigid: every Lipschitz horizontal path based at the origin occurs along a suitable sequence for a bounded horizontal input. Their nilpotent developments may retain signed area and higher area-moment information at leading Carnot scale. The period array records this leading development on its diagonal and the successive subleading defects below it.

The remainder of this introduction gives the background for these results and fixes the conventions used throughout the paper.  Section~\ref{sec:nilpotent-development} develops horizontal BV paths in graded nilpotent groups.  Section~\ref{sec:macroscopic-development} proves the asymptotic-development theorem and the rigidity and realization results for macroscopic paths.  Section~\ref{sec:geometry-first-layers} interprets the first three homogeneous layers geometrically.  Section~\ref{sec:closed-trajectories} proves the finite-time and Carnot-scale nilpotent \(\pi_1\)-de Rham theorems.  Section~\ref{sec:path-defects} develops the period array, recovers the Benardete--Mitchell correction in step two, and exhibits the first free 3-step obstruction. The continuity estimates used in the asymptotic-development theorem are collected in Appendix~\ref{app:continuity}.
\subsection*{Acknowledgements:} The author would like to thank Tristan Collins for useful comments on an earlier draft of this paper. 

\subsection{Background}
\label{sec:background}
The asymptotic study of a long trajectory begins with displacement.  After a path segment of duration \(T\) has been closed by a uniformly bounded path, Schwartzman's asymptotic cycle records the limiting homology class after division by \(T\) \cite{Schw57}.  This is the correct abelian invariant, but it cannot distinguish two trajectories whose increments have the same average and occur in a different chronological order.  In a nilpotent target, that order is visible.  If a horizontal path first follows \(X\) and then \(Y\), its developed endpoint is \(\exp(Y)\exp(X)\); reversing the order gives \(\exp(X)\exp(Y)\).  The logarithms have the same first-layer displacement but opposite depth-two commutator terms.  At depth three they also retain information about how the commutator interacts with the two horizontal directions.

This observation suggests that the natural object is not an asymptotic group element, but instead is the long path itself, viewed at macroscopic scale, together with its nilpotent development.  Three related paths occur throughout the paper.  There is an original trajectory in a manifold, its horizontal primitive in the first layer of a nilpotent Lie algebra, and the corresponding group-valued developed path.  These three levels should not be conflated.  The original trajectory carries the topology of the ambient manifold, the horizontal primitive retains the chronological order of the increments in a linear space, and the developed path converts that ordered information into the noncommutative geometry of the nilpotent group.

Let
\[
  \g=V_1\oplus\cdots\oplus V_n
\]
be a graded nilpotent Lie algebra generated by \(V_1\), and let \(G\) be its connected, simply connected Lie group.  For a continuous horizontal path \(\ell:[0,\infty)\to V_1\) of bounded variation on compact intervals, write
\[
  H[t]=H_1[t]+\cdots+H_n[t]
\]
for the Chacon--Fomenko Lie integral \cite{CF91a} of \(\ell|_{[0,t]}\), and set \(U(t)=\exp H[t]\).  Thus, \(U\) is the nilpotent development of \(\ell\).  The construction is finite because the Lie algebra is nilpotent, so the logarithm is a finite Lie polynomial in the ordered increments of the path.

The path segment observed up to time \(T\) is placed on the unit interval by
\begin{equation}
  \ell_T(u)=\frac{\ell(Tu)-\ell(0)}{T},
  \qquad 0\leq u\leq1.
  \label{eq:intro-rescaled-path}
\end{equation}
This normalization retains the shape of the path, not only its endpoint average.  The algebraic foundation of the paper is the identity
\begin{equation}
  H\bigl[\ell_T|_{[0,u]}\bigr]
  =\delta_{1/T}H[Tu],
  \qquad 0\leq u\leq1,
  \label{eq:intro-exact-path-scaling}
\end{equation}
where \(\delta_r\) denotes the Carnot dilation.  Thus, path rescaling before development agrees exactly with Carnot dilation after development.

The asymptotic-development theorem, Theorem~\ref{thm:macroscopic-path}, combines this identity with continuity of the nilpotent development on uniformly bounded families of continuous BV paths.  Suppose that \(T_j\to\infty\), that \(\ell_{T_j}\) converges uniformly to a path \(\ell^\infty\), and that the variations of the rescaled paths are uniformly bounded.  Then the complete developed paths satisfy
\begin{equation}
  \delta_{1/T_j}U(T_j u)
  \longrightarrow
  U^\infty(u)
  :=\exp H\bigl[\ell^\infty|_{[0,u]}\bigr]
  \label{eq:intro-developed-path-limit}
\end{equation}
uniformly for \(0\leq u\leq1\).  In particular,
\[
  \delta_{1/T_j}U(T_j)
  \longrightarrow
  \exp H[\ell^\infty].
\]
The limiting point in the asymptotic cone is the endpoint of the developed macroscopic path.  The endpoint is a consequence of the path-level convergence rather than the primary object of the construction.

The limiting horizontal path may be non-straight when it is obtained along a subsequence.  Its higher homogeneous logarithmic components can then be nonzero and record chronological information that survives at Carnot scale.  This subsequential freedom is complete within the natural regularity class: every Lipschitz path in \(V_1\) based at the origin occurs as a macroscopic horizontal path along some sequence for a bounded measurable horizontal input.  By contrast, if \(\ell_T\) converges as \(T\to\infty\) without passage to a subsequence, the scaling relation among the family \(\ell_T\) forces the limit to be a straight segment.  The higher layers of its development vanish.  This distinction between subsequential realization and full-limit rigidity plays a central role.  The leading Carnot geometry is determined by the limiting macroscopic path; in the full-limit regime, the nonabelian information moves below the leading scale and appears in the manner by which the rescaled paths approach their limiting segment.

The first three homogeneous layers make the path interpretation concrete.  The first layer is horizontal displacement.  The second is signed area, with the sign determined by the Chacon--Fomenko Lie integral later--earlier ordering convention.  The third layer is a corrected logarithmic first moment of area.  Because of nilpotency, path development in the Heisenberg group ends with the area term, while in the Engel Lie group, it ends with the first area moment.  These examples illustrate what the asymptotic-development theorem says geometrically.

Topology enters after the open-path geometry has been understood.  A nilpotent representation of \(\pi_1(M)\) is realized by the holonomy of a flat Lie-algebra-valued development form.  The nilpotent \(\pi_1\)-de Rham correspondence identifies the Lie integral of a finite based loop with the logarithm of its holonomy.  An open long trajectory must be completed by a closing path before it represents a homotopy class.  A closing whose developed logarithm remains bounded disappears at Carnot scale, so the asymptotic-cone endpoint obtained from the macroscopic path is independent of the closing.  This gives the leading asymptotic \(\pi_1\)-de Rham statement proved in Theorem~\ref{thm:carnot-pi1-derham}.

The asymptotic cone does not retain every nonabelian feature of a long path.  Fix a scale sequence \(\mathbf T=(T_j)\) with \(T_j\to\infty\).  When successive limits exist along this sequence, one obtains expansions
\[
  H_k[T_j]=\sum_{m=1}^{k}T_j^m\Lamn{k}{m}(\mathbf T)+o(T_j).
\]
The triangular collection of coefficients \(\Lamn{k}{m}(\mathbf T)\) is the sequential Chacon--Fomenko period array along \(\mathbf T\).  Its diagonal is determined by the macroscopic path arising along that sequence and may be nonzero in higher layers when the limiting path is nonstraight.  In the full-limit regime, scaling rigidity forces the macroscopic path to be straight; for a bounded input with an ordinary mean, every higher diagonal entry then vanishes.  The entries below the diagonal measure successive defects in the convergence of the complete homogeneous components along the chosen sequence.  They are generally more sensitive than the Carnot limit.  In particular, bounded closings can affect them.  In depth two, Benardete and Mitchell identified the nonlinear correction that removes the first closing dependence and produces a homotopy invariant \cite{BM93}.  In depth three, nested BCH interactions create additional terms, and the free 3-step Lie algebra on three generators shows why no unmodified depth-two argument can suffice.

The principal contributions of this paper are as follows.  We introduce a Carnot-scale asymptotic theory for nilpotent Lie developments based on the Chacon--Fomenko Lie integral.  This yields a nilpotent analogue of the Benardete--Mitchell asymptotic homotopy invariant, establishes a Carnot-scale \(\pi_1\)-de Rham correspondence for flat nilpotent connections, and identifies the diagonal of the resulting period array with the Lie development of the macroscopic horizontal path.  We further show that, although full macroscopic limits are necessarily straight, every Lipschitz horizontal path based at the origin occurs as a subsequential macroscopic limit of a bounded horizontal input, revealing a fundamental distinction between full and sequential asymptotic behaviour.  Together, these results establish the foundations of a Carnot-scale asymptotic homotopy theory for nilpotent Lie groups.

\section{Nilpotent development of horizontal paths}
\label{sec:nilpotent-development}

The purpose of this section is to define the nilpotent development used throughout this paper and to isolate the features that make it compatible with macroscopic rescaling.  The Chacon--Fomenko construction is used with its chosen chronological convention: later increments multiply on the left and occupy the leftmost positions in the ordered bracket integrals.

\subsection{Graded nilpotent groups and Carnot dilations}
Let
\[
    \g=V_1\oplus\cdots\oplus V_n
\]
be a finite-dimensional graded \(n\)-step nilpotent Lie algebra:
\[
    [V_i,V_j]\subseteq V_{i+j},
\]
where \(V_k=\{0\}\) for \(k>n\).  We assume that \(\g\) is generated by
\(V_1\).  Thus, a bracket of \(k\) horizontal elements lies in \(V_k\).

For \(r>0\), the Carnot dilation \(\delta_r:\g\to\g\) is defined by
\[
    \delta_r(v_1+\cdots+v_n)
    =
    rv_1+r^2v_2+\cdots+r^nv_n,
    \qquad v_k\in V_k.
\]
Because the grading is compatible with the bracket,
\[
    \delta_r[x,y]=[\delta_rx,\delta_ry].
\]
Let \(G\) be the connected, simply connected Lie group with Lie algebra
\(\g\).  Since \(\exp:\g\to G\) is a global diffeomorphism, each
\(\delta_r\) integrates to a group automorphism, also denoted by
\(\delta_r\), characterized by
\[
    \delta_r(\exp X)=\exp(\delta_rX).
\]

Fix a norm \(\|\cdot\|\) on \(\g\), and use its restriction on
each graded subspace \(V_k\).  Since \(\g\) is finite dimensional and
the Lie bracket is continuous, there is a constant \(C\geq 1\) such that
\[
    \|[x,y]\|\leq C\|x\|\,\|y\|
\]
for all \(x,y\in\g\).

\subsection{The Carnot group and its asymptotic cone}
\label{sec:carnot-asymptotic-cone}

Let \(e\) denote the identity element of \(G\).  Equip \(V_1\) with the
restriction of the fixed norm and let \(d_{\mathrm{CC}}\) be the
corresponding left-invariant Carnot--Carath\'eodory distance on \(G\).
Because \(V_1\) generates \(\g\), this distance is finite, and the topology
it induces agrees with the manifold topology on \(G\).  The group
dilations satisfy
\begin{equation}
  d_{\mathrm{CC}}(\delta_rg,\delta_rh)
  =r\,d_{\mathrm{CC}}(g,h),
  \label{eq:CC-homogeneity}
\end{equation}
and so
\[
  \delta_{1/T}:
  \bigl(G,e,T^{-1}d_{\mathrm{CC}}\bigr)
  \longrightarrow
  \bigl(G,e,d_{\mathrm{CC}}\bigr)
\]
is a pointed isometry for every \(T>0\).  In this graded setting, the
Carnot asymptotic cone is represented canonically by the same
group \(G\); no limiting change of underlying group is required.

If
\[
  g_T=\exp H[T]\in G
\]
is the endpoint of a developed path, then its representative in the
blow-down at scale \(T\) is
\begin{equation}
  g_T^{\mathrm{cone}}
  :=\delta_{1/T}g_T
  =\exp\!\left(\delta_{1/T}H[T]\right)
  =\exp\!\left(\sum_{k=1}^n\frac{H_k[T]}{T^k}\right).
  \label{eq:cone-representative}
\end{equation}
Thus, the layerwise Carnot-normalized logarithm gives coordinates for a
point in the group-level asymptotic cone.

The ambient space here is the simply connected nilpotent group in which
the lifted paths and their developed group endpoints live.  If the original dynamics take place on a compact quotient \(G/\Gamma\),
equip \(G/\Gamma\) with any fixed Riemannian metric \(d_M\).  Since the
quotient is compact, it has finite diameter, and
\[
  \operatorname{diam}\bigl(G/\Gamma,T^{-1}d_M\bigr)
  =
  T^{-1}\operatorname{diam}\bigl(G/\Gamma,d_M\bigr)
  \longrightarrow0.
\]
Thus, the rescaled compact quotient itself collapses to a point.  The
nontrivial asymptotic information is obtained by lifting or developing
the long path into \(G\) before taking the Carnot blow-down.

\begin{remark}[The graded restriction]
For a simply connected nilpotent group that is not already graded, the
metric asymptotic cone is described by the associated graded Carnot group;
this is the setting of Pansu's asymptotic-cone theorem~\cite{Pan83}.  The
nilpotent \(\pi_1\)-de Rham correspondence in Section~\ref{sec:exact-pi1-derham} applies before such
a grading is chosen.  The asymptotic results of this paper are formulated
in the graded target itself, so that the cone and its dilations are
available without an additional passage to an associated graded group.
\end{remark}

\subsection{Path-ordered development and logarithmic coordinates}

Let \(K:[a,b]\to\g\) be an integrable Lie-algebra-valued input.  The group-valued equation
\begin{equation}
  U'(t)=K(t)U(t),\qquad U(a)=e,
  \label{eq:path-ordered-ode}
\end{equation}
has the path-ordered solution
\[
  U(b)=\mathcal P\exp\!\int_a^bK(t)\,dt.
\]
Along this solution, the pullback of the universal right Maurer--Cartan form is
\[
  U^*(dg\,g^{-1})=K(t)\,dt.
\]
Since the exponential map of a connected, simply connected nilpotent group is a global diffeomorphism, there is a unique logarithm
\[
  H[a,b]=\log U(b)\in\g.
\]
The Chacon--Fomenko Lie integral \cite{CF91a} is this logarithm, written
\[
  H[a,b]=\Lint_a^bK(t)\,dt.
\]
Passing to the logarithm does not approximate the development.  In a nilpotent target it expresses the developed endpoint as a finite Lie polynomial and places it in the graded vector space on which the Carnot dilations act.

If one path \(c_1\) is traversed first and a path \(c_2\) second, the chronological convention gives
\begin{equation}
  \exp H[c_2*c_1]=\exp H[c_2]\exp H[c_1],
  \qquad
  H[c_2*c_1]=\operatorname{BCH}(H[c_2],H[c_1]).
  \label{eq:CF-concatenation-BCH}
\end{equation}
The convention is the later factor appears on the left, which controls the signs in the area and area-moment calculations below.

\subsection{Continuous bounded-variation paths}

\begin{definition}
\label{def:BV-path}
A path \(\gamma:[a,b]\to\g\) has \emph{bounded variation} if
\[
    \Var(\gamma;[a,b])
    :=
    \sup_{\mathcal P}
    \sum_{j=1}^{m}
    \|\gamma(t_j)-\gamma(t_{j-1})\|
    <\infty,
\]
where the supremum is over all partitions
\(\mathcal P:a=t_0<t_1<\cdots<t_m=b\).  
\end{definition}

A path satisfying Definition~\ref{def:BV-path} will be called a
\emph{bounded-variation path}, or a \emph{BV path}. 

In this paper, the horizontal paths
take values in the graded subspace \(V_1\subseteq\g\), and we restrict ourselves to continuous BV paths.  This avoids any convention concerning
the order in which a jump is traversed and is the natural level of generality for the key asymptotic-development theorem, Theorem~\ref{thm:macroscopic-path}.  In particular, all absolutely continuous paths arising from time-dependent horizontal inputs are included.

As well, all integrals against a continuous BV path are
Riemann--Stieltjes integrals.  Equip
\(\operatorname{End}(\g)\) with the operator norm induced by the fixed
norm on \(\g\).  If
\(A:[a,b]\to\operatorname{End}(\g)\) is continuous, choose a tagged partition
\[
    \mathcal P:\quad a=t_0<t_1<\cdots<t_m=b,
    \qquad \xi_j\in[t_{j-1},t_j],
\]
and form the sums
\[
    \sum_{j=1}^m
    A(\xi_j)\bigl(\gamma(t_j)-\gamma(t_{j-1})\bigr).
\]
As the mesh of \(\mathcal P\) tends to zero, these sums converge in
\(\g\).  Their limit is denoted by
\[
    \int_a^b A(t)\,d\gamma(t).
\]
The integrands arising in the Chacon--Fomenko recursion are of this
form.  For example, if \(H:[a,b]\to\g\) is continuous, then
\[
    \int_a^b [H(t),d\gamma(t)]
    :=
    \int_a^b \ad_{H(t)}\,d\gamma(t),
\]
where \(\ad_{H(t)}(X)=[H(t),X]\).

Thus, the symbol \(d\gamma\) records the increments of the path in the
Riemann--Stieltjes sum.  If \(\gamma\) is absolutely continuous, then
\[
    \int_a^b A(t)\,d\gamma(t)
    =
    \int_a^b A(t)\gamma'(t)\,dt.
\]
We shall use the standard estimate
\[
    \left\|\int_a^bA(t)\,d\gamma(t)\right\|
    \leq
    \sup_{a\leq t\leq b}\|A(t)\|\,
    \Var(\gamma;[a,b]).
\]

\subsection{Ordered iterated integrals}

Following Chen's theory of iterated path integrals \cite{Chen77}, although
the eventual quantities are Lie-algebra valued, it is useful first to retain
the chronological ordering of the path increments in the tensor algebra
\[
  T(V_1)=\bigoplus_{q\geq0}V_1^{\otimes q}.
\]
Equip each tensor power \(V_1^{\otimes q}\) with the projective
tensor norm induced by the fixed norm on \(V_1\).

Let \(\gamma:[0,1]\to V_1\) be a continuous BV path.  For \(q\geq1\),
define its \(q\)-th ordered tensor integral over \([s,t]\) by
\begin{equation}
  I_q(\gamma)_{s,t}
  :=
  \int_{s<u_q<\cdots<u_1<t}
  d\gamma(u_1)\otimes\cdots\otimes d\gamma(u_q)
  \in V_1^{\otimes q},
  \label{eq:ordered-tensor-simplex}
\end{equation}
where the latest parameter occupies the leftmost tensor slot.  For
\(q=1\), this gives
\[
  I_1(\gamma)_{s,t}=\gamma(t)-\gamma(s).
\]
Equivalently, for \(q\geq2\),
\begin{equation}
  I_q(\gamma)_{s,t}
  =
  \int_s^t
  \Bigl(
    d\gamma(u)\otimes I_{q-1}(\gamma)_{s,u}
  \Bigr).
  \label{eq:ordered-tensor-recursion}
\end{equation}
Here the entire tensor
\(d\gamma(u)\otimes I_{q-1}(\gamma)_{s,u}\) is the integrand; in
particular, the upper endpoint \(u\) in the second factor varies with the
integration variable.  Since \(\gamma\) is continuous and of bounded
variation, these finite-dimensional Riemann--Stieltjes integrals are well
defined.

Set \(B_1(v)=v\).  For \(q\geq2\), define the left-nested bracket on
pure tensors by
\begin{equation}
  B_q(v_1\otimes\cdots\otimes v_q)
  :=
  [\cdots[[v_1,v_2],v_3],\ldots,v_q].
  \label{eq:left-nested-bracket-map}
\end{equation}
By multilinearity, these maps extend uniquely to linear maps
\[
  B_q:V_1^{\otimes q}\longrightarrow V_q.
\]
The Chacon--Fomenko bracket integral of order \(j\) is then
\begin{equation}
  T_j[\gamma]
  :=
  B_{j+1}\bigl(I_{j+1}(\gamma)_{0,1}\bigr).
  \label{eq:Tj-from-tensor-integral}
\end{equation}
Thus, the ordered tensor integral retains the chronological tensor data,
while \(B_{j+1}\) extracts the left-nested bracket contribution relevant
to the construction.  In expanded form,
\begin{equation}
  T_j[\gamma]
  =
  \int_{0<u_{j+1}<\cdots<u_1<1}
  [\cdots[[d\gamma(u_1),d\gamma(u_2)],d\gamma(u_3)],
    \ldots,d\gamma(u_{j+1})].
  \label{eq:Tj-expanded-bracket-integral}
\end{equation}
In particular,
\[
  T_0[\gamma]=\gamma(1)-\gamma(0).
\]

The ordering is the Chacon--Fomenko convention: later increments occupy the leftmost slots
and all brackets are left-nested.  In particular,
\begin{equation}
  T_1[\gamma]
  =
  \int_{0<u_2<u_1<1}[d\gamma(u_1),d\gamma(u_2)],
  \label{eq:T1-ordered-bracket-integral}
\end{equation}
which is compatible with the product ordering for \(U'(t)=K(t)U(t)\).

\subsection{The Chacon--Fomenko Lie integral of a BV path}

The Chacon--Fomenko Lie integral may be defined for a continuous BV
horizontal path by the same finite recursion used for an integrable
horizontal input.

Let \(b_{2p}\) be the Bernoulli numbers and put
\[
    k_{2p}=\frac{b_{2p}}{(2p)!}.
\]

\begin{definition}
For a continuous BV path \(\gamma:[0,1]\to V_1\), define
\[
    H[\gamma]
    =
    H_1[\gamma]+\cdots+H_n[\gamma],
    \qquad H_k[\gamma]\in V_k,
\]
recursively as follows:
\[
    H_1[\gamma]=T_0[\gamma],
\]
and, for \(k\geq 1\),
\begin{align}
    (k+1)H_{k+1}[\gamma]
    ={}&
    T_k[\gamma]
    +
    \sum_{r=1}^{k}
    \Bigg\{
       \frac12[H_r[\gamma],T_{k-r}[\gamma]]
       \notag\\
       &\quad+
       \sum_{\substack{p\geq 1\\2p\leq r}}
       k_{2p}
       \sum_{\substack{m_i>0\\m_1+\cdots+m_{2p}=r}}
       [H_{m_1}[\gamma],
          [\cdots[
             H_{m_{2p}}[\gamma],
             T_{k-r}[\gamma]
          ]\cdots]]
    \Bigg\}.
    \label{eq:CF-recursion}
\end{align}
Terms of degree greater than \(n\) vanish because \(\g\) is
\(n\)-step nilpotent.  The resulting finite sum \(H[\gamma]\) is the
\emph{Chacon--Fomenko Lie integral} of \(\gamma\), denoted also by
\[
    \Lint_\gamma d\gamma.
\]
\end{definition}

When
\[
    \gamma(t)=\int_0^tK(s)\,ds
\]
is absolutely continuous, \(d\gamma=K(t)\,dt\), and this definition
agrees with the usual Chacon--Fomenko Lie integral
\[
    \Lint_0^1K(t)\,dt.
\]

Every term in \(H_k[\gamma]\) is homogeneous of degree \(k\) in the
increments of \(\gamma\).  Equivalently,
\[
    H_k[r\gamma]=r^kH_k[\gamma]
\]
for every scalar \(r\).

\begin{remark}[Chronological ordering and the area sign]
Three orderings must remain consistent.  The partition points are indexed chronologically, the product exponential places later factors on the left, and the simplex variables satisfy \(u_1\geq u_2\geq\cdots\), with the latest increment in the leftmost bracket slot.  Thus, the depth-two component in the convention used here is the negative of the L\'evy area in the standard probabilistic convention, which orders the bracket as earlier--later.  This is a convention of orientation and bracket order.
\end{remark}

Because every term in \(H_k[\gamma]\) is homogeneous of degree \(k\) in the increments of \(\gamma\), the path functional satisfies
\begin{equation}
  H_k[r\gamma]=r^kH_k[\gamma].
  \label{eq:homogeneity-path-functional}
\end{equation}
Nilpotence makes the construction finite.  In an \(n\)-step target, every bracket of length greater than \(n\) vanishes, and the Lie integral terminates exactly at \(H_n\).  There is no truncation error and no convergence issue for the Lie series itself.  The limiting questions in this paper concern the long-time behaviour of a finite collection of homogeneous components.

\section{Macroscopic development of long paths}
\label{sec:macroscopic-development}

We now compare long path segments on a common parameter interval.  The key point is that the rescaling of the horizontal path and the Carnot dilation of its development agree exactly.  The only analytic input is continuity of the finite nilpotent development under uniform convergence with uniform variation control.  After proving the asymptotic-development theorem, we characterize the subsequential macroscopic paths generated by bounded inputs and contrast their freedom with the rigidity of full limits.

\subsection{Continuity of developed paths}

For a continuous horizontal BV path \(\gamma:[0,1]\to V_1\), define its developed path by
\begin{equation}
  U_\gamma(u)=\exp H\bigl[\gamma|_{[0,u]}\bigr],
  \qquad 0\leq u\leq1.
  \label{eq:developed-BV-path}
\end{equation}
The next proposition is proved in Appendix~\ref{app:continuity}.  Its path-level formulation is the form needed for the asymptotic-development theorem, Theorem~\ref{thm:macroscopic-path}.

\begin{proposition}[Uniform continuity of nilpotent development]
\label{prop:pathwise-CF-continuity}
Let \(\gamma_j,\gamma:[0,1]\to V_1\) be continuous BV paths such that \(\gamma_j\to\gamma\) uniformly and
\[
  \sup_j\Var(\gamma_j;[0,1])<\infty.
\]
Then \(\gamma\) is a BV path and, for every \(1\leq k\leq n\),
\[
  \sup_{0\leq u\leq1}
  \left\|
    H_k\bigl[\gamma_j|_{[0,u]}\bigr]
    -H_k\bigl[\gamma|_{[0,u]}\bigr]
  \right\|
  \longrightarrow0.
\]
Consequently, \(U_{\gamma_j}\to U_\gamma\) uniformly on \([0,1]\) in the topology of \(G\).
\end{proposition}

\subsection{Exact scaling of the developed path}

We use the same path functional on any compact interval.  It depends only on path increments and is invariant under increasing reparameterization.

Let \(\ell:[0,\infty)\to V_1\) be continuous and of bounded variation on compact intervals.  Put
\[
  H[t]=H\bigl[\ell|_{[0,t]}\bigr],
  \qquad
  U(t)=\exp H[t].
\]
For \(T>0\), define the rescaled path by \eqref{eq:intro-rescaled-path}.

\begin{proposition}[Exact path-scaling identity]
\label{prop:exact-path-scaling}
For every \(T>0\), every \(0\leq u\leq1\), and every \(1\leq k\leq n\),
\begin{equation}
  H_k\bigl[\ell_T|_{[0,u]}\bigr]
  =\frac{H_k[Tu]}{T^k}.
  \label{eq:exact-path-scaling-components}
\end{equation}
Equivalently,
\begin{equation}
  H\bigl[\ell_T|_{[0,u]}\bigr]
  =\delta_{1/T}H[Tu],
  \qquad
  U_{\ell_T}(u)=\delta_{1/T}U(Tu).
  \label{eq:exact-path-scaling-group}
\end{equation}
\end{proposition}

\begin{proof}
For \(u\in[0,1]\),
\[
  \ell_T(u)=\frac{\ell(Tu)-\ell(0)}{T}.
\]
Since the Lie integral depends only on path increments and is invariant under increasing reparameterization,
\[
  H_k\bigl[\ell_T|_{[0,u]}\bigr]
  =\frac1{T^k}H_k\bigl[\ell|_{[0,Tu]}\bigr]
  =\frac{H_k[Tu]}{T^k}.
\]
Summing the homogeneous components and using \(\delta_r(\exp X)=\exp(\delta_rX)\) gives \eqref{eq:exact-path-scaling-group}.
\end{proof}

\subsection{The asymptotic-development theorem}

\begin{definition}[Macroscopic horizontal path]
\label{def:macroscopic-horizontal-path}
Let \(T_j\to\infty\).  A continuous BV path \(\ell^\infty:[0,1]\to V_1\) is a \emph{macroscopic horizontal path along \((T_j)\)} if \(\ell_{T_j}\to\ell^\infty\) uniformly and
\[
  \sup_j\Var(\ell_{T_j};[0,1])<\infty.
\]
Its nilpotent development
\[
  U^\infty(u)=\exp H\bigl[\ell^\infty|_{[0,u]}\bigr]
\]
is the \emph{developed macroscopic path}.  If the same convergence and variation bound hold for the full family as \(T\to\infty\), then \(\ell^\infty\) is called the \emph{full macroscopic horizontal path}.
\end{definition}

\begin{theorem}[Asymptotic development of paths]
\label{thm:macroscopic-path}
Let \(\ell:[0,\infty)\to V_1\) be continuous and of bounded variation on compact intervals, and let \(T_j\to\infty\).  Suppose that \(\ell^\infty\) is a macroscopic horizontal path along \((T_j)\).  Then
\begin{equation}
  \delta_{1/T_j}U(T_j u)
  \longrightarrow
  U^\infty(u)
  \label{eq:uniform-developed-path-limit}
\end{equation}
uniformly for \(0\leq u\leq1\).  Equivalently, for every homogeneous layer,
\begin{equation}
  \sup_{0\leq u\leq1}
  \left\|
    \frac{H_k[T_j u]}{T_j^k}
    -H_k\bigl[\ell^\infty|_{[0,u]}\bigr]
  \right\|
  \longrightarrow0.
  \label{eq:uniform-homogeneous-path-limit}
\end{equation}
In particular,
\begin{equation}
  \delta_{1/T_j}H[T_j]\longrightarrow H[\ell^\infty],
  \qquad
  \delta_{1/T_j}U(T_j)\longrightarrow\exp H[\ell^\infty].
  \label{eq:endpoint-asymptotic-development}
\end{equation}
The final group element is the endpoint of the developed macroscopic path and represents the limiting point of the developed endpoints in the Carnot asymptotic cone along \((T_j)\).
\end{theorem}

\begin{proof}
Proposition~\ref{prop:exact-path-scaling} identifies the left side of \eqref{eq:uniform-developed-path-limit} with the development of \(\ell_{T_j}\).  Proposition~\ref{prop:pathwise-CF-continuity} identifies the uniform limit of those developed paths with the development of \(\ell^\infty\).  The componentwise statement follows from the same argument before exponentiation, and the endpoint statements are obtained by setting \(u=1\).
\end{proof}

The asymptotic-development theorem distinguishes the macroscopic horizontal path, its developed path, and the endpoint of that development.  In general, the endpoint does not determine the path that produced it.  A non-straight limit may retain chronological information in its higher homogeneous components even when the horizontal endpoint alone does not reveal that information.

\subsection{Full-limit rigidity}

\begin{proposition}[Rigidity of full macroscopic paths]
\label{prop:full-macroscopic-rigidity}
If \(\ell_T\to\ell^\infty\) uniformly as \(T\to\infty\), then
\[
  \ell^\infty(u)=u\mu,
  \qquad 0\leq u\leq1,
\]
where \(\mu=\ell^\infty(1)\).
\end{proposition}

\begin{proof}
For \(u>0\), the rescaled paths satisfy \(\ell_T(u)=u\ell_{Tu}(1)\).  Since \(Tu\to\infty\), full convergence gives \(\ell^\infty(u)=u\ell^\infty(1)\).  The identity at \(u=0\) is immediate.
\end{proof}

The preceding proposition is a rigidity statement about convergence of the full scaling family.  For subsequential limits of bounded inputs, there is no corresponding restriction on shape beyond the regularity forced by boundedness.

\begin{proposition}[Realization of subsequential macroscopic paths]
\label{prop:subsequential-realization}
Let \(\eta:[0,1]\to V_1\).  The following are equivalent:
\begin{enumerate}
\item \(\eta\) is Lipschitz and \(\eta(0)=0\).
\item There exists a bounded measurable horizontal input
\[
  K\in L^\infty([0,\infty);V_1)
\]
and a sequence \(T_j\to\infty\) such that, for
\[
  \ell(t)=\int_0^t K(s)\,ds,
\]
the path \(\eta\) is a macroscopic horizontal path along \((T_j)\).
\end{enumerate}
\end{proposition}

\begin{proof}
Suppose first that \(K\in L^\infty([0,\infty);V_1)\), and let \(M=\lVert K\rVert_\infty\).  For \(0\leq u\leq v\leq1\),
\[
  \lVert\ell_T(v)-\ell_T(u)\rVert
  =\frac1T\left\lVert\int_{Tu}^{Tv}K(s)\,ds\right\rVert
  \leq M(v-u).
\]
Thus, every \(\ell_T\) is \(M\)-Lipschitz and vanishes at the origin.  The same is true of every uniform subsequential limit, and
\[
  \Var(\ell_T;[0,1])\leq M.
\]
This yields that (2) implies (1).

Conversely, suppose that \(\eta\) is \(L\)-Lipschitz and \(\eta(0)=0\).  Choose \(T_j\to\infty\) so rapidly that
\[
  r_j:=\frac{T_{j-1}}{T_j}\longrightarrow0,
  \qquad
  2T_{j-1}<T_j
\]
for \(j\geq2\).  Define \(\ell\) first on \([0,T_1]\) by
\[
  \ell(t)=T_1\eta(t/T_1).
\]
Suppose inductively that \(\ell\) has been defined through time \(T_{j-1}\), with
\[
  \ell(T_{j-1})=T_{j-1}\eta(1).
\]
On \([T_{j-1},2T_{j-1}]\), let \(\ell\) be the affine path joining \(T_{j-1}\eta(1)\) to \(T_j\eta(2r_j)\), and on \([2T_{j-1},T_j]\), put
\[
  \ell(t)=T_j\eta(t/T_j).
\]
The speed of the affine joining path is at most
\[
  \frac{\lVert T_j\eta(2r_j)-T_{j-1}\eta(1)\rVert}{T_{j-1}}
  \leq3L,
\]
while the remaining part of the new block is \(L\)-Lipschitz.  Hence, \(\ell\) is globally \(3L\)-Lipschitz.  It is absolutely continuous.  Choose a measurable representative \(K\) of its almost-everywhere derivative, setting \(K=0\) on the exceptional null set.  Then
\[
  \ell(t)=\int_0^tK(s)\,ds,
  \qquad
  \lVert K\rVert_\infty\leq3L.
\]

At scale \(T_j\), the construction gives
\[
  \ell_{T_j}(u)=\eta(u),
  \qquad 2r_j\leq u\leq1.
\]
For \(0\leq u\leq2r_j\), the global Lipschitz bound and \(\eta(0)=0\) give
\[
  \lVert\ell_{T_j}(u)-\eta(u)\rVert
  \leq3Lu+Lu
  \leq8Lr_j.
\]
Thus, \(\ell_{T_j}\to\eta\) uniformly.  Moreover,
\[
  \Var(\ell_{T_j};[0,1])\leq\lVert K\rVert_\infty,
\]
so \(\eta\) is a macroscopic horizontal path along \((T_j)\).
\end{proof}

Thus, full macroscopic horizontal limits are rigidly straight, whereas subsequential macroscopic horizontal limits generated by bounded inputs are otherwise unrestricted within the class of Lipschitz paths based at the origin.  The realizing primitive is absolutely continuous and hence remains within the class of continuous paths of bounded variation on compact intervals used in Theorem~\ref{thm:macroscopic-path}.

\begin{corollary}[Straight-line vanishing]
\label{cor:straight-line}
If \(\ell^\infty(u)=u\mu\), then
\[
  H_1[\ell^\infty]=\mu,
  \qquad
  H_k[\ell^\infty]=0\quad(k\geq2).
\]
Consequently, whenever a full macroscopic horizontal path exists, the higher Carnot-normalized endpoint components vanish.
\end{corollary}

\begin{proof}
Every increment of a straight path is parallel to \(\mu\), so every bracket containing two or more increments vanishes.  The first homogeneous component is the endpoint displacement.
\end{proof}

\begin{corollary}[Bounded input with an ordinary mean]
\label{cor:bounded-mean-straight}
Let \(K:[0,\infty)\to V_1\) be bounded and let \(\ell(t)=\int_0^tK(s)\,ds\).  If
\[
  \frac1T\int_0^TK(s)\,ds\longrightarrow\mu,
\]
then \(\ell_T\to(u\mapsto u\mu)\) uniformly, the variations of \(\ell_T\) are uniformly bounded, and
\[
  \frac{H_1[T]}T\longrightarrow\mu,
  \qquad
  \frac{H_k[T]}{T^k}\longrightarrow0\quad(k\geq2).
\]
\end{corollary}

\begin{proof}
For \(u\in[0,1]\), write \(\ell_T(u)=u\,\ell_{Tu}(1)\).  The convergence of the averages is uniform away from \(u=0\), while boundedness of \(K\) controls the remaining short interval near \(u=0\).  Moreover, \(\Var(\ell_T)\leq\|K\|_\infty\).  The conclusion follows from Theorem~\ref{thm:macroscopic-path} and Corollary~\ref{cor:straight-line}.
\end{proof}
\section{The Geometry of the First Three Homogeneous Layers}
\label{sec:geometry-first-layers}

This section identifies the geometric content of the first three
homogeneous components before turning to their subleading asymptotics.  For a horizontal path, \(H_1\) records displacement, \(H_2\)
records signed area, and \(H_3\) records a corrected first moment of area
in the available depth-three bracket directions.  The geometric object is
always the complete homogeneous component \(H_k\), not an isolated
simplex integral taken out of the recursion.  We also give explicit
macroscopic horizontal paths for which the higher components do not
vanish.  A piecewise-linear corner displays the memory of chronological order, a
circular arc makes the depth-two area visible, and the Engel specialization
shows how the same corner produces a depth-three area moment.

\subsection{The complete first layers}
\label{sec:first-layer-geometry}

Let
\[
  K(t)=\sum_{i=1}^r f_i(t)X_i\in V_1,
  \qquad
  \ell(t)=\int_0^tK(s)\,ds.
\]
The first homogeneous component is
\[
  H_1[t]=\ell(t)-\ell(0),
\]
so it records horizontal displacement.  The next simplex term is
\[
  T_1[t]
  =\int_{0<u_2<u_1<t}[K(u_1),K(u_2)]\,du_2\,du_1,
\]
and the Chacon--Fomenko recursion~\eqref{eq:CF-recursion} gives
\[
  H_2[t]=\frac12T_1[t].
\]
Thus, each coordinate of \(H_2\) is an ordered signed area in a horizontal
coordinate plane, with the sign fixed by the Chacon--Fomenko later--earlier ordering.

At depth three the geometric quantity is no longer the triple-simplex
term by itself.  The recursion combines
\[
  T_2[t]
  =\int_{0<u_3<u_2<u_1<t}
    [[K(u_1),K(u_2)],K(u_3)]\,du_3\,du_2\,du_1
\]
with brackets involving \(H_1\) and \(H_2\).  The resulting complete
component \(H_3[t]\) is the corrected logarithmic first area moment.  The
Engel recursion in Section~\ref{sec:engel-recursion}, especially
equation~\eqref{eq:H3-Engel}, makes the correction explicit: the complete
depth-three component contains not only the  triple-simplex term but
also the displacement--area bracket term
\(\frac1{12}F(t)A(t)X_4\), where \(F(t)\) is the horizontal
\(X_1\)-displacement, \(A(t)\) is the depth-two area integral, and
\(X_4\) spans the third layer.  The Chacon-Fomenko nonabelian Stokes calculation in
Section~\ref{sec:stokes-heis-engel}, especially
equation~\eqref{eq:Engel-open-closed-correction}, identifies the same
correction geometrically as the BCH term relating the closed first area
moment to the logarithm of the open path.  This distinction between a raw
ordered integral and the complete layer becomes more, not less,
important when subleading asymptotics are studied.

\subsection{A corner and chronological memory}
\label{sec:corner-memory}

Let \(X,Y\in V_1\), and let \(a,b>0\).  Consider the horizontal
piecewise-linear path
\begin{equation}
  \gamma_{a,b}^{X,Y}(s)
  =
  \begin{cases}
    sX, & 0\leq s\leq a,\\[2mm]
    aX+(s-a)Y, & a\leq s\leq a+b.
  \end{cases}
  \label{eq:corner-path}
\end{equation}
The path first follows the direction \(X\) and then turns into the
direction \(Y\).  With the Chacon--Fomenko chronological convention, its developed
endpoint and Lie integral are
\begin{equation}
  g_{\gamma}=\exp(bY)\exp(aX),
  \qquad
  H[\gamma_{a,b}^{X,Y}]
  =\operatorname{BCH}(bY,aX).
  \label{eq:corner-development}
\end{equation}
The first three homogeneous components are
\begin{align}
  H_1[\gamma_{a,b}^{X,Y}]&=aX+bY,
  \label{eq:corner-H1}\\
  H_2[\gamma_{a,b}^{X,Y}]&=-\frac{ab}{2}[X,Y],
  \label{eq:corner-H2}\\
  H_3[\gamma_{a,b}^{X,Y}]&=
  \frac{a^2b}{12}[X,[X,Y]]
  +\frac{ab^2}{12}[Y,[Y,X]].
  \label{eq:corner-H3}
\end{align}
Thus, \(H_1\) records only the total horizontal displacement, while
\(H_2\) records the fact that the \(X\)-segment occurred before the
\(Y\)-segment, the orientation of the turn, and the signed area swept out by the path in the \([X,Y]\)-direction. The depth-three component records how the second-order contribution created by the corner is itself affected by motion in the \(X\)- and \(Y\)-directions. Notice that it detects information that disappears when \([X,Y]\) is central. 

If \([X,Y]=0\), then the two actions commute and
\[
  \exp(bY)\exp(aX)=\exp(aX+bY).
\]
In that case, every higher component vanishes.  The corner is 
remembered only when its two directions generate noncommutative motion.
The corner \(\gamma_{a,b}^{X,Y}\) is Lipschitz and based at the origin.  Proposition~\ref{prop:subsequential-realization} gives a bounded horizontal input and a sequence \(T_j\to\infty\) for which it is the macroscopic horizontal path.  Theorem~\ref{thm:macroscopic-path} then gives
\[
  \frac{H_k[T_j]}{T_j^k}
  \longrightarrow H_k[\gamma_{a,b}^{X,Y}],
\]
so the chronological memory in \eqref{eq:corner-H2} and
\eqref{eq:corner-H3} survives directly in the corresponding Carnot
layers of the limiting point in the asymptotic cone.

\subsection{The nonabelian Stokes formula}

Chacon and Fomenko's nonabelian Stokes theorem
\cite[Definition~6, p.~266; Theorem~8, pp.~280--281]{CF91b} uses the Chacon--Fomenko
curvature
\[
  \mathcal F=\widehat d(-\mathcal A)=-d{\mathcal A} + {\mathcal A} \wedge {\mathcal A}
\]
of a Lie-algebra-valued connection form \(\mathcal A\).  Let
\(\Phi:I^2\to M\) be a smooth parameterized surface, oriented by
\(du\wedge dv\), and take \(\Phi(0,0)\) as the base point.  For each
\((u,v)\in I^2\), let \(g_{(u,v)}\) be the image under \(\Phi\) of the
right-angle path that runs from \((0,0)\) to \((0,v)\) and then to
\((u,v)\), and let \(P(u,v)\) be parallel transport along this path.
The twisted curvature is the base-point-valued function
\begin{equation}
  \widetilde{\mathcal F}^{*}(u,v)
  =\operatorname{Ad}_{P(u,v)^{-1}}
    \left(
      (\Phi^*\mathcal F)_{(u,v)}(\partial_u,\partial_v)
    \right).
  \label{eq:twisted-curvature-definition}
\end{equation}
In a matrix representation, this is
\(P(u,v)^{-1}(\Phi^*\mathcal F)(\partial_u,\partial_v)P(u,v)\).
The conjugation is essential: infinitesimal curvature holonomies based at
different points of the surface must be transported to a common base
point before they are multiplied noncommutatively.  The surface Lie
integral also retains the ordering induced by the parameter square and
the Chacon--Fomenko subdivision; in the abelian case both the conjugation
and the ordering become immaterial.

With \(-\mathcal A^*\) denoting the boundary input, the nonabelian Stokes theorem gives the
group-level identity
\begin{equation}
  \exp\!\left(\Lint_{\partial\Phi}-\mathcal A^*\right)
  =
  \exp\!\left(L\!\iint_{I^2}\widetilde{\mathcal F}^{*}\right).
  \label{eq:stokes}
\end{equation}
The equality of the two Lie-algebra-valued exponents is equation~(29) in
\cite[p.~281]{CF91b}, under the local convergence and injectivity
conditions required there for a general finite-dimensional Lie
algebra.  In the setting of this paper, those restrictions disappear:
\(\g\) is nilpotent, so both Chacon--Fomenko series terminate, and
\(\exp:\g\to G\) is a global diffeomorphism.  Applying its global inverse turns \eqref{eq:stokes} into an equality of finite
\(\g\)-valued Lie integrals with no local smallness condition.

The Chacon--Fomenko nonabelian Stokes theorem concerns the Lie integral of the \emph{whole} oriented
boundary path.  If that boundary is decomposed into several edges, their
contributions are combined by the generalized chronological BCH product
\eqref{eq:CF-concatenation-BCH}; they are not added as vectors in
\(\g\).

\subsection{Parameterized fillings and Chacon--Fomenko Stokes}
\label{sec:simplices-stokes}

For a pair of horizontal generators \((X_i,X_j)\), put
\[
  \omega_{ij}=X_i\,dx_i+X_j\,dx_j,
  \qquad
  \mathcal A_{ij}=-\omega_{ij}.
\]
Following Chacon and Fomenko's curvature convention~\cite[Section~2]{CF91b}, the curvature is
\[
  \mathcal F_{ij}=\widehat d(-\mathcal A_{ij}).
\]
With the component convention used in this paper,
\[
  \mathcal F_{ij}=-[X_i,X_j] \,dx_i\wedge dx_j.
\]
For a smooth parameterized filling \(\Phi:I^2\to\mathbb R^2\), write
\(\omega_{ij}^{*}\) for the boundary input induced by \(\omega_{ij}\).
The nilpotent specialization of the Chacon--Fomenko nonabelian Stokes theorem gives
\begin{equation}
  \Lint_{\partial\Phi}\omega_{ij}^{*}
  =L\!\iint_{I^2}\widetilde{\mathcal F}_{ij}^{*}.
  \label{eq:CF-Stokes-pair}
\end{equation}
The integral on the right is taken over the parameter
square because \(\widetilde{\mathcal F}_{ij}^{*}\) denotes the twisted
pullback of the curvature form by \(\Phi\).  The left side is the Lie integral of the
entire oriented boundary.  If the boundary edges, in chronological order,
are \(c_1,\ldots,c_m\), then
\[
  \Lint_{\partial\Phi}\omega_{ij}^{*}
  =\log\!\left(e^{H[c_m]}\cdots e^{H[c_1]}\right).
\]
Thus, the boundary pieces combine by the ordered generalized BCH product.

\subsubsection*{The ruled filling}
Let \(r:[0,t]\to\mathbb R^2\) be the projection of the horizontal
primitive \(\ell\) onto the \((X_i,X_j)\)-coordinate plane, written in
those coordinates, and assume \(r(0)=0\). Define the outward chord
\[
  c_t(s)=\frac{s}{t}r(t),
  \qquad 0\leq s\leq t.
\]
Its reverse, which closes the orbit segment from \(r(t)\) to the origin,
is
\[
  \overline c_t(s)=c_t(t-s).
\]
Let
\begin{equation}
  R_t(s)=r(s)-c_t(s).
  \label{eq:Rt-deviation}
\end{equation}
and use the ruled map
\begin{equation}
  \Phi_t(s,v)=c_t(s)+vR_t(s),
  \qquad (s,v)\in[0,t]\times[0,1].
  \label{eq:ruled-surface-filling}
\end{equation}
With orientation \(ds\wedge dv\), the induced boundary traverses the
outward chord and then the orbit in reverse.  If \(H_\gamma\) and
\(H_c\) denote the Lie integrals of the outward orbit path \(r\) and the
outward chord \(c_t\), respectively, then
\begin{equation}
  H_{\partial\Phi_t}
  =\operatorname{BCH}(-H_\gamma,H_c),
  \label{eq:ruled-boundary-BCH}
\end{equation}
and so
\begin{equation}
  H_\gamma
  =\operatorname{BCH}(H_c,-H_{\partial\Phi_t}).
  \label{eq:open-from-ruled-boundary}
\end{equation}
The chord has only a first-layer logarithm, but its BCH interaction with
the surface term can contribute in higher layers.  This is the origin of
the Engel open-path correction in
\eqref{eq:Engel-open-closed-correction}.

\begin{remark}[The parameter square and rank drops]
Chacon and Fomenko carry out the construction on the parameter square
\cite[Sections~4--6]{CF91b}, using
pullbacks of the connection and curvature, right-angle transport paths,
and a snake-enumerated subdivision into small rectangles.  These operations do not
require the map \(\Phi\) to have an inverse.  Where the rank drops, the
pulled-back curvature two-form vanishes because the two parameter tangent
vectors are dependent.  The same parameter-square proof applies
to the ruled maps used here.  In the nilpotent target, the Chacon--Fomenko series
terminates and \(\exp:\g\to G\) is a global diffeomorphism, so the
resulting Lie-algebra identity is global.
\end{remark}

For later estimates, it is useful that the ruled surface quantities reduce
to one-parameter integrals.  From~\eqref{eq:ruled-surface-filling},
\[
  \partial_s\Phi_t=c_t'(s)+vR_t'(s),
  \qquad
  \partial_v\Phi_t=R_t(s),
\]
and hence
\begin{equation}
  \det D\Phi_t(s,v)
  =\det\bigl(c_t'(s)+vR_t'(s),R_t(s)\bigr)
  =J_0(s)+vJ_1(s),
  \label{eq:ruled-jacobian}
\end{equation}
where
\[
  J_0(s)=\det(c_t'(s),R_t(s)),
  \qquad
  J_1(s)=\det(R_t'(s),R_t(s)).
\]
Equivalently,
\[
  \Phi_t^*(dx\wedge dy)
  =\bigl(J_0(s)+vJ_1(s)\bigr)\,ds\wedge dv.
\]
Let \(\Sigma_t\) denote the oriented parameterized \(2\)-chain determined
by \(\Phi_t\), and define integration over it by pullback:
\begin{equation}
  \iint_{\Sigma_t}dx\wedge dy
  :=\int_{[0,t]\times[0,1]}\Phi_t^*(dx\wedge dy).
  \label{eq:Sigma-pullback-definition}
\end{equation}
Then
\begin{equation}
  \iint_{\Sigma_t}dx\wedge dy
  =\int_0^t\left(J_0(s)+\frac12J_1(s)\right)ds.
  \label{eq:area-ruled-one-parameter}
\end{equation}
If \(c_{t,x}\) and \(R_{t,x}\) denote the first coordinates of
\(c_t\) and \(R_t\), then
\begin{align}
  \iint_{\Sigma_t}x\,dx\wedge dy
  &=\int_0^t\Bigl[
    c_{t,x}(s)J_0(s)
    +\frac12c_{t,x}(s)J_1(s) \notag\\
  &\hspace{3.6em}
    +\frac12R_{t,x}(s)J_0(s)
    +\frac13R_{t,x}(s)J_1(s)
    \Bigr]ds.
  \label{eq:xmoment-ruled-one-parameter}
\end{align}
These identities express the signed area and first \(x\)-moment of the
ruled filling as one-parameter integrals determined by the deviation
\(R_t\) of the projected horizontal path from its chord.  They will be
combined with the ordered Chacon--Fomenko Stokes expansion and the boundary BCH
identity to recover the complete logarithmic components of the open path.

\subsection{Heisenberg geometry: signed area}
\label{sec:Heisenberg}

Let \(\mathfrak h=\operatorname{span}\{X,Y,Z\}\) with
\([X,Y]=Z\), and let \(H\) be the simply connected Heisenberg group.  In
first-kind exponential coordinates,
\begin{equation}
  (x,y,z)(x',y',z')
  =\left(x+x',y+y',z+z'+\frac12(xy'-yx')\right).
  \label{eq:Heis-product}
\end{equation}
Fix a lattice \(\Gamma\subset H\).  For a horizontal input
\begin{equation}
  K(s)=f(s)X+g(s)Y,
  \label{eq:Heis-K}
\end{equation}
write
\[
  F(t)=\int_0^tf(s)\,ds,
  \qquad
  G(t)=\int_0^tg(s)\,ds.
\]
The Lie integral, equivalently the logarithm of the horizontal
development, is
\begin{equation}
  H[t]=F(t)X+G(t)Y+\frac12A(t)Z,
  \label{eq:Heis-Omega}
\end{equation}
where
\begin{equation}
  A(t)=\int_{0<u_2<u_1<t}
  \bigl(f(u_1)g(u_2)-g(u_1)f(u_2)\bigr)\,du_2\,du_1.
  \label{eq:Heis-area}
\end{equation}
Equivalently,
\begin{equation}
  A(t)=\int_0^t f(u)G(u)\,du-\int_0^t g(u)F(u)\,du.
  \label{eq:Heis-area-antiderivative}
\end{equation}

The nonabelian Stokes calculation identifies \(\frac12A(t)Z\) with the signed area
of the horizontal path
\[
  r(u)=(F(u),G(u))\in V_1\cong\mathbb R^2
\]
closed by \(\overline c_t\), including the BCH interaction with the closing
chord.  With the Chacon--Fomenko later--earlier ordering, \(H_2[t]\) is the negative
of L\'evy area in the standard probabilistic orientation convention.

\begin{example}[A circular macroscopic horizontal arc]
\label{ex:heis-circular-arc}
For \(R>0\) and \(0<\theta\leq2\pi\), consider the horizontal path
\begin{equation}
  \gamma_{R,\theta}(u)
  =R\sin(\theta u)X
   +R\bigl(1-\cos(\theta u)\bigr)Y,
  \qquad 0\leq u\leq1.
  \label{eq:circular-arc-path}
\end{equation}
Its velocity has constant magnitude \(R\theta\), while its horizontal
direction turns through the angle \(\theta\).  Formula
\eqref{eq:Heis-area-antiderivative} gives
\begin{align*}
  A[\gamma_{R,\theta}]
  &=R^2\int_0^1\theta\bigl(\cos(\theta u)-1\bigr)\,du\\
  &=-R^2\bigl(\theta-\sin\theta\bigr),
\end{align*}
and hence
\begin{equation}
  H_2[\gamma_{R,\theta}]
  =-\frac{R^2}{2}\bigl(\theta-\sin\theta\bigr)Z.
  \label{eq:circular-arc-H2}
\end{equation}
The quantity
\[
  \frac{R^2}{2}\bigl(\theta-\sin\theta\bigr)
\]
is the ordinary oriented area enclosed by the arc and its reverse chord.
Thus, the depth-two homogeneous component records the negative of that area in the
orientation convention used here.  For a full circle, \(\theta=2\pi\),
we have
\[
  H_1[\gamma_{R,2\pi}]=0,
  \qquad
  H_2[\gamma_{R,2\pi}]=-\pi R^2Z.
\]
The horizontal endpoint has returned to the origin, but the developed
endpoint retains the area swept out by the rotating horizontal direction.
Since \(\gamma_{R,\theta}\) is Lipschitz and based at the origin, Proposition~\ref{prop:subsequential-realization} realizes it as the macroscopic horizontal path of a bounded input along some sequence \(T_j\to\infty\).  Along that sequence,
\[
  \frac{H_2[T_j]}{T_j^2}
  \longrightarrow
  -\frac{R^2}{2}\bigl(\theta-\sin\theta\bigr)Z.
\]
\end{example}

The sequence \(T_j\) is not merely selecting points along the trajectory. Each \(T_j\) determines a new blow-down of the entire path on \([0,T_j]\), and along the chosen subsequence, these rescaled paths converge to the same circular arc. The reader may find it helpful to work through the rescaling explicitly for \(T_j=j!\), for which the contribution of the preceding stage occupies a vanishing fraction \(T_{j-1}/T_j=1/j\) of the macroscopic time interval. 

For a bounded horizontal input with ordinary mean
\(\mu=\bar fX+\bar gY\), the rescaled horizontal path converges to the
straight segment \(u\mapsto u\mu\).  Therefore,
\begin{equation}
  \delta_{1/t}H[t]
  =\frac{F(t)}{t}X+\frac{G(t)}{t}Y+\frac{A(t)}{2t^2}Z
  \longrightarrow \bar fX+\bar gY,
  \label{eq:Heis-Carnot-normalized}
\end{equation}
and in particular
\[
  \frac{A(t)}{t^2}\longrightarrow0.
\]
This means the first possible nonzero Heisenberg defect is the linear
area rate \(A(t)/t\).  Its existence is not implied by the ordinary mean;
orbit-endpoint terms and ordered correlations enter at precisely that scale.
The existence of this linear area rate is a separate asymptotic question.
Section~\ref{sec:path-defects} places it in the period array, while
Section~\ref{sec:BM-correction} explains the closing correction required
for a homotopy interpretation.

\begin{example}[Closed commutator-loop check in Heisenberg]
\label{ex:heis-commutator-loop}
Assume that \(X,Y,Z\) belong to a Mal'cev basis compatible with
\(\Gamma\)~\cite{Mal49}, normalized so that \(\exp X,\exp Y,\exp Z\in\Gamma\).  For
\(N\in\mathbb N\), the group commutator
\[
  \eta_N=\exp(NX)\exp(NY)\exp(-NX)\exp(-NY)
\]
has
\[
  \log\eta_N=N^2Z,
  \qquad
  \delta_{1/N}\log\eta_N=Z.
\]
Since \(\eta_N=\exp(N^2Z)\in\Gamma\), the corresponding path projects
under \(H\to H/\Gamma\) to a closed loop.  This verifies directly that the Lie integral is the global logarithm
of the nilpotent fundamental-group element and that its second Carnot
layer is the commutator class.
\end{example}

\subsection{Engel geometry: the first moment of area}
\label{sec:engel}

Let \(\ee\) be the Engel Lie algebra with basis
\(X_1,X_2,X_3,X_4\) and brackets
\[
  [X_1,X_2]=X_3,
  \qquad
  [X_1,X_3]=X_4,
\]
all other basis brackets being zero.  Its stratification is
\[
  V_1=\operatorname{span}\{X_1,X_2\},
  \qquad
  V_2=\mathbb RX_3,
  \qquad
  V_3=\mathbb RX_4,
\]
with Carnot dilation
\[
  \delta_\lambda(x_1X_1+x_2X_2+x_3X_3+x_4X_4)
  =\lambda x_1X_1+\lambda x_2X_2+\lambda^2x_3X_3+\lambda^3x_4X_4.
\]
Let \(E\) be the simply connected Engel Lie group and fix a lattice compatible
with this rational structure.  The asymptotic calculations take place in
first-kind logarithmic coordinates, where the exponential map and Carnot
dilations are global.

\subsubsection{The Engel recursion}
\label{sec:engel-recursion}

For
\[
  K(s)=f(s)X_1+g(s)X_2,
  \qquad
  F(t)=\int_0^tf(s)\,ds,
  \qquad
  G(t)=\int_0^tg(s)\,ds,
\]
we have
\[
  H_1[t]=F(t)X_1+G(t)X_2.
\]
Define
\begin{equation}
  A(t)=\int_{0<u_2<u_1<t}
  \bigl(f(u_1)g(u_2)-g(u_1)f(u_2)\bigr)\,du_2\,du_1.
  \label{eq:EngelA}
\end{equation}
Then
\begin{equation}
  H_2[t]=\frac12A(t)X_3.
  \label{eq:H2-Engel}
\end{equation}
The next ordered simplex coefficient is
\begin{equation}
  B(t)=\int_{0<u_3<u_2<u_1<t}
  f(u_3)\bigl(g(u_1)f(u_2)-f(u_1)g(u_2)\bigr)
  \,du_3\,du_2\,du_1.
  \label{eq:EngelB}
\end{equation}
The depth-three recursion gives
\begin{equation}
  H_3[t]
  =\left(\frac13B(t)+\frac1{12}F(t)A(t)\right)X_4.
  \label{eq:H3-Engel}
\end{equation}
This formula is the first place where the distinction between a raw
simplex term and the complete layer is unavoidable.  The depth-three
geometry receives contributions both from \(B(t)\) and
from the Chacon--Fomenko recursion term coupling the horizontal displacement
\(H_1[t]\) with the depth-two area component \(H_2[t]\), namely the
displacement--area bracket term \(\frac1{12}F(t)A(t)X_4\).

\subsubsection{Chacon--Fomenko Stokes and the Engel moment}
\label{sec:stokes-heis-engel}

Write \(X=X_1\), \(Y=X_2\), \(Z=X_3\), and \(W=X_4\).  For the pair form
\[
  \omega=X\,dx+Y\,dy,
  \qquad
  \mathcal A=-\omega,
\]
the curvature is
\[
  \mathcal F=\widehat d(-\mathcal A)=-Z\,dx\wedge dy.
\]
In the Heisenberg quotient \(Z\) is central, so the surface term is just
signed area.  In the Engel Lie algebra, write \(P_g\) for parallel transport
along the standard right-angle path.  Then
\[
  P_g^{-1}ZP_g=Z-xW,
\]
and hence
\begin{equation}
  \mathcal F^*=-Z\,dx\wedge dy+xW\,dx\wedge dy.
  \label{eq:Engel-twisted-curvature}
\end{equation}
For the closed planar loop obtained by projecting the horizontal path to
\(V_1\cong\mathbb R^2\) and adjoining the reverse chord, write
\(\Sigma\) for the corresponding oriented filling.  Then
\begin{equation}
  z_{\mathrm{cl}}=-\iint_\Sigma dx\wedge dy,
  \qquad
  w_{\mathrm{cl}}=\iint_\Sigma x\,dx\wedge dy.
  \label{eq:Engel-closed-moments}
\end{equation}
Thus, the second layer is signed area and the third layer is its first
\(x\)-moment.

For an open horizontal segment with logarithm
\[
  H[t]=FX+GY+zZ+wW,
\]
closing by the straight horizontal chord and applying 3-step BCH gives
\begin{equation}
  w_{\mathrm{cl}}=w-\frac12Fz,
  \qquad
  w=w_{\mathrm{cl}}+\frac12Fz.
  \label{eq:Engel-open-closed-correction}
\end{equation}
This yields
\begin{equation}
  w=\iint_\Sigma\left(x-\frac12F(t)\right)dx\wedge dy.
  \label{eq:Engel-open-moment-formula}
\end{equation}
Thus, the complete Engel coordinate is a corrected first moment of
signed area.  The correction is the BCH
term required to pass from the closed boundary moment to the open-path
logarithm, and it is the geometric counterpart of the
\(F(t)A(t)/12\) term in \eqref{eq:H3-Engel}.

\begin{example}[The corner in the Engel group]
\label{ex:engel-corner}
Specialize the corner \eqref{eq:corner-path} to \(X=X_1\) and
\(Y=X_2\).  Since \([Y,[Y,X]]=0\) in the Engel Lie algebra,
\eqref{eq:corner-development}--\eqref{eq:corner-H3} give
\begin{equation}
  \log(e^{bY}e^{aX})
  =aX+bY-\frac12abZ+\frac1{12}a^2bW.
  \label{eq:Engel-corner-BCH}
\end{equation}
Thus,
\[
  H_2[\gamma_{a,b}^{X,Y}]= -\frac12abZ,
  \qquad
  H_3[\gamma_{a,b}^{X,Y}]= \frac1{12}a^2bW.
\]
The closing triangle has area \(ab/2\) and first \(x\)-moment
\(a^2b/3\).  Formula~\eqref{eq:Engel-open-closed-correction} gives
\[
  w=\frac13a^2b+\frac12a\left(-\frac12ab\right)
   =\frac1{12}a^2b,
\]
in agreement with BCH.  The same turn that creates a nonzero signed area
at depth two creates a nonzero corrected first area moment at
depth three.  Proposition~\ref{prop:subsequential-realization} realizes this corner as a macroscopic horizontal path of a bounded input along some sequence \(T_j\to\infty\), and both quantities survive on their corresponding Carnot layers.
\end{example}

\begin{example}[Engel commutator check]
\label{ex:engel-commutator-loop}
For
\[
  \eta_T=\exp(TX)\exp(TY)\exp(-TX)\exp(-TY),
\]
the 3-step BCH formula gives
\[
  \log\eta_T=T^2Z+\frac12T^3W,
  \qquad
  \delta_{1/T}\log\eta_T=Z+\frac12W.
\]
This closed-loop calculation independently verifies that once the
second-layer commutator is noncentral, the same fundamental-group endpoint
already carries third-layer Carnot data.
\end{example}

\subsection{Carnot normalization and defect geometry}
\label{sec:first-layers-defect-geometry}

For the rescaled horizontal path
\[
  \ell_T(u)=\frac{\ell(Tu)-\ell(0)}{T},
  \qquad 0\leq u\leq1,
\]
homogeneity gives the identity
\[
  H_k[\ell_T]=\frac{H_k[T]}{T^k}.
\]
If \(T_j\to\infty\) and \(\ell_{T_j}\to\ell^\infty\) uniformly with
uniformly bounded variation, then
\[
  \frac{H_k[T_j]}{T_j^k}\longrightarrow H_k[\ell^\infty].
\]
The corner and circular-arc examples give realized bounded-input families for which these limits are
nonzero in the higher Carnot layers because the macroscopic horizontal
path along the chosen subsequence retains noncommuting chronological increments.  The Engel corner
shows that depth-three data can survive at the same leading Carnot scale.

A limit of the full family as \(T\to\infty\) is necessarily straight by
Proposition~\ref{prop:full-macroscopic-rigidity}.  In particular, for bounded
horizontal inputs with an ordinary mean,
\(\ell^\infty(u)=u\mu\), and so
\[
  \frac{H_k[T]}{T^k}\longrightarrow0,
  \qquad k\geq2.
\]
Thus, the nonstraight examples above are not merely intrinsic horizontal path
models: Proposition~\ref{prop:subsequential-realization} realizes each of them as a macroscopic horizontal path of a bounded input along a subsequence.  They cannot be
full macroscopic horizontal paths under this normalization.  In the full-limit regime, the
vanishing of the higher diagonal does not eliminate the area and
area-moment geometry.  It says that those quantities live in the manner
by which the Carnot-scaled endpoint approaches the limiting point
\(\exp\mu\) in the asymptotic cone.

The period-array regime studies this approach along a chosen scale
sequence through integral powers of \(T_j^{-1}\).  More generally, a first
defect along such a sequence may have the form
\[
  \frac{H_k[T_j]}{T_j^k}-H_k[\ell^\infty]
  =\varepsilon_jD_k+o(\varepsilon_j),
  \qquad D_k\in V_k,
\]
for a scale \(\varepsilon_j\to0\) that need not be an integral power of
\(T_j^{-1}\).  The Carnot degree \(k\) specifies the graded direction in
which the defect lives; the sequence \(\varepsilon_j\) specifies how rapidly
the corresponding logarithmic coordinate approaches its limiting value.
In the full-limit regime, this specializes to the approach toward the
straight-path value \(H_k[\ell^\infty]=0\) for \(k\geq2\).

The ruled formulas make this distinction concrete.  If the horizontal
velocity is bounded and
\[
  \sup_{0\leq s\leq t}\lVert R_t(s)\rVert=O(t^\alpha),
  \qquad 0\leq\alpha<1,
\]
then \eqref{eq:area-ruled-one-parameter} and
\eqref{eq:xmoment-ruled-one-parameter} give the geometric bounds
\[
  H_2[t]=O(t^{1+\alpha}),
  \qquad
  H_3[t]=O(t^{2+\alpha}),
\]
subject to the cancellations built into the complete homogeneous components.  The
bounded-deviation case \(\alpha=0\) is the scale at which a linear area
term and a quadratic Engel moment naturally appear.  Fractional or
logarithmically modified deviation rates lead to corresponding non-power
defect scales.

Section~\ref{sec:path-defects} focuses on the structured case in
which the successive defects occur in integral powers and form the
period array.  The present paper does not impose general dynamical
hypotheses guaranteeing the existence of these limits.  The finite-scale geometry
established here remains valid whether the defect expansion is polynomial,
non-power, oscillatory, or only subsequential.

\section{Flat development and closed trajectories}
\label{sec:closed-trajectories}

The preceding sections concern open horizontal paths and their nilpotent developments.  We now introduce the topological setting in which a developed endpoint becomes the holonomy of a loop.  The order is deliberate: the asymptotic-development theorem is independent of the choice of a manifold or a fundamental-group representation, while the topological statement arises when an open trajectory is closed and its development is identified with flat holonomy.
This finite-dimensional nilpotent correspondence belongs to the broader de Rham theory of Mal'cev completion; see, for example, Hain \cite{Hain98}.
\subsection{The nilpotent \texorpdfstring{$\pi_1$}{pi1} de Rham correspondence}
\label{sec:exact-pi1-derham}

Let \(M\) be a connected smooth manifold with base point \(x_0\), and
let
\[
  p:\widetilde M\longrightarrow M
\]
be its universal covering projection.  Fix a lift \(\widetilde x_0\) of
\(x_0\), and use the right action of \(\pi_1(M,x_0)\) on
\(\widetilde M\).  Let \(G\) be a connected, simply connected nilpotent
Lie group, and write \(\g\) for its Lie algebra.  Since \(G\) is
connected, simply connected, and nilpotent, the exponential map
\[
  \exp:\g\longrightarrow G
\]
is a global diffeomorphism, and
\(\log:G\to\g\) denotes its globally defined inverse.

Let
\[
  \theta^R=dg\,g^{-1}
\]
be the right Maurer--Cartan form.  For a \(\g\)-valued one-form
\(\mathcal B\), write
\[
  \widehat d\mathcal B=d\mathcal B+\mathcal B\wedge\mathcal B.
\]
The curvature of a \(\g\)-valued connection form
\(\mathcal A\) is
\[
  \mathcal F=\widehat d(-\mathcal A).
\]
Flatness means \(\widehat d(-\mathcal A)=0\).

If \(\gamma:[a,b]\to M\) is a path,
\(\widetilde\gamma:[a,b]\to\widetilde M\) is a lift of this path, and
\(g=D\circ\widetilde\gamma\) is the developed trajectory, then the
pullback of the right Maurer--Cartan form satisfies
\[
  g^*(dg\,g^{-1})=K(t)\,dt,
\]
where \(K(t)\) is the Lie-algebra-valued function appearing in the
Chacon--Fomenko differential equation \cite{CF91a}.  Thus, \(K(t)\) is  the
coefficient of the pullback of the universal right Maurer--Cartan form
along the developed trajectory.

\begin{theorem}[Based nilpotent \(\pi_1\) de Rham correspondence]
\label{thm:nilpotent-pi1-derham}
There is a natural correspondence between homomorphisms
\[
  \rho:\pi_1(M,x_0)\longrightarrow G
\]
and based gauge-equivalence classes of flat \(\g\)-valued connection
forms on \(M\), where the based gauge transformations are required to be the
identity at \(x_0\).  More precisely, a representation \(\rho\) admits a
normalized smooth developing map
\[
  D:\widetilde M\longrightarrow G,
  \qquad
  D(\widetilde x\cdot\gamma)=D(\widetilde x)\rho(\gamma),
  \qquad
  D(\widetilde x_0)=e,
\]
for which the form \(D^*\theta^R\), invariant under deck
transformations, descends to a flat form
\(\mathcal A\in\Omega^1(M;\g)\).  Conversely, every flat form
\(\mathcal A\) has a unique normalized developing map satisfying
\[
  D^*\theta^R=p^*\mathcal A,
  \qquad D(\widetilde x_0)=e,
\]
and its equivariance homomorphism is the holonomy representation of
\(\mathcal A\).

For every absolutely continuous path \(c:[0,1]\to M\) with \(c(0)=x_0\),
and for the lift \(\widetilde c\) with
\(\widetilde c(0)=\widetilde x_0\), the Lie integral satisfies the
exact endpoint identity
\begin{equation}
  \exp\!\left(\Lint_c\mathcal A\right)
  =D(\widetilde c(1))D(\widetilde c(0))^{-1}.
  \label{eq:exact-endpoint}
\end{equation}
If \(c\) is a based loop and \([c]\in\pi_1(M,x_0)\) is its based
homotopy class, then
\begin{equation}
  \exp\!\left(\Lint_c\mathcal A\right)=\rho([c]),
  \qquad
  \Lint_c\mathcal A=\log\rho([c]).
  \label{eq:exact-pi1-derham}
\end{equation}

\end{theorem}

\begin{proof}
Given \(\rho\), form the associated flat principal \(G\)-bundle.  A
connected, simply connected nilpotent group is diffeomorphic to its Lie
algebra and hence contractible, so this bundle admits a smooth
trivialization.  After normalizing at \(x_0\), the corresponding section
gives the \(\rho\)-equivariant map \(D\).  Right invariance of
\(\theta^R\) makes \(D^*\theta^R\) invariant under the right deck action,
so it descends uniquely to a form \(\mathcal A\) on \(M\).  The
Maurer--Cartan equation gives \(\widehat d(-\mathcal A)=0\).

Conversely, let \(\mathcal A\) be flat.  Then \(p^*\mathcal A\) is flat
on \(\widetilde M\).  Since \(\widetilde M\) is simply connected, the global
Maurer--Cartan integration theorem gives a unique smooth map \(D\),
normalized by \(D(\widetilde x_0)=e\), such that
\(D^*\theta^R=p^*\mathcal A\).  For each deck transformation \(\gamma\),
the maps \(\widetilde x\mapsto D(\widetilde x\cdot\gamma)\) and
\(\widetilde x\mapsto D(\widetilde x)\) have the same right logarithmic
derivative.  So, they differ by a constant right factor
\(\rho(\gamma)\):
\[
  D(\widetilde x\cdot\gamma)=D(\widetilde x)\rho(\gamma).
\]
The right-action convention makes \(\rho\) a homomorphism, and this is
the holonomy representation of \(\mathcal A\).  Conversely, a normalized
\(\rho\)-equivariant developing map has deck-invariant pullback
\(D^*\theta^R\), which descends uniquely to the corresponding flat form.

Along a lifted path, let
\[
  K_c(t)=\mathcal A_{c(t)}(\dot c(t)),
\]
defined almost everywhere; since \(c\) is absolutely continuous and
\(\mathcal A\) is smooth, \(K_c\) is integrable.  The lift
\(\widetilde c\) is absolutely continuous, so
\[
  U(t)=D(\widetilde c(t))D(\widetilde c(0))^{-1}
\]
is absolutely continuous and solves \(U'(t)=K_c(t)U(t)\) almost
everywhere, with \(U(0)=e\).  By uniqueness for this linear equation
with integrable coefficient, \(U\) is the path-ordered exponential of
\(K_c\).  The Lie integral of \(K_c(t)\) is the logarithm of
\(U(1)\), proving \eqref{eq:exact-endpoint}.  If \(c\) is a based loop, then
\(\widetilde c(1)=\widetilde x_0\cdot[c]\), and equivariance gives
\(U(1)=\rho([c])\).  Applying the global inverse
\(\log:G\to\g\) proves \eqref{eq:exact-pi1-derham}.  Changing the
trivialization by a gauge transformation that is the identity at
\(x_0\) preserves \(\rho\); allowing a nontrivial value at the base point
conjugates it.
\end{proof}

Thus, Theorem~\ref{thm:nilpotent-pi1-derham} interprets the Lie
integral as the logarithmic integration of the universal
right Maurer--Cartan form along the developed trajectory in \(G\).

The nilpotent \(\pi_1\)-de Rham correspondence is a finite-time topological statement:
it identifies the logarithmic endpoint of every finite based loop before
any long-time normalization is introduced.  The asymptotic problem begins
only after Carnot dilation is applied to the finite-time endpoints of long paths
or of long orbit segments completed by controlled closings.  The nilpotent \(\pi_1\)-de Rham
correspondence and the asymptotic-development theorem thus play different
roles.  They are combined only in
Theorem~\ref{thm:carnot-pi1-derham}, after bounded-closing invariance has
been established and after the compatibility between the flat development
and the horizontal input has been stated explicitly.  The additional
requirements for a full asymptotic correspondence are discussed in
Section~\ref{sec:conclusion}.

\subsection{Full flat development forms and horizontal inputs}
\label{sec:full-versus-horizontal}

Suppose now that \(\g=V_1\oplus\cdots\oplus V_n\) is graded.  A full flat
development form decomposes as
\[
  \mathcal A=\mathcal A_1+\cdots+\mathcal A_n,
  \qquad
  \mathcal A_j\in\Omega^1(M;V_j).
\]
The degree-one part of flatness says that \(\mathcal A_1\) is closed; the
higher degree equations couple \(d\mathcal A_j\) to brackets of the lower
components.  In particular, the higher components of a flat nilpotent
connection are generally not closed scalar forms considered separately.

Most explicit calculations in this paper begin instead with a prescribed
horizontal input
\[
  K(t)=\omega_{\gamma(t)}(\dot\gamma(t))\in V_1,
  \qquad
  \omega\in\Omega^1(M;V_1).
\]
In a dynamical model, the same input may be obtained from a continuous
function \(F:X\to V_1\) on a state space carrying a flow \(\phi^t\), by
setting \(K_x(t)=F(\phi^t x)\).  Such an input may be the first layer of a
flat development form, but it need not be the evaluation of the full form
\(\mathcal A\).  Its Lie integral is still the logarithmic
endpoint of the corresponding developed path, equivalently the logarithm of \(U(1)\), where \(U:[0,T]\to G\) is the unique solution of 
\[
U'(t)=K_x(t)U(t), \qquad U(0)=e.
\]  
What is not
automatic is the topological identification of that endpoint with
\(\log\rho([c])\).  The notation \([c]\in\pi_1(M,x_0)\) applies only when
\(c\) is a based loop, and the identity
\(\Lint_c\mathcal A=\log\rho([c])\) requires the full flat development
form, or an explicitly compatible path-holonomy construction.

This distinction is essential later.  The nilpotent
\(\pi_1\)-de Rham correspondence concerns flat development forms and their
holonomy.  The period array and defect profiles are defined for chosen
horizontal paths and scale sequences whose chronological increments are encoded by the Chacon--Fomenko
recursion.  In applications, the two structures are linked.
Example~\ref{ex:natural-class} exhibits the basic class in which
they coincide by construction: horizontal trajectories on a compact
nilmanifold developed by the natural flat form.

\subsection{Endpoint ordering and the curvature convention}
\label{sec:endpoint-ordering-curvature}

For \(X,Y\in\g\), the Baker--Campbell--Hausdorff expression
\[
  \operatorname{BCH}(X,Y)=\log\!\bigl(\exp X\,\exp Y\bigr)
\]
is a finite Lie polynomial.  The chronological concatenation convention is the one fixed in equation~\eqref{eq:CF-concatenation-BCH}: if \(c_1\) is traversed first and \(c_2\) second, the later factor multiplies on the left.

The flat right Maurer--Cartan development form in
Theorem~\ref{thm:nilpotent-pi1-derham} has zero curvature.  The
nonzero curvature calculations elsewhere in the paper concern auxiliary
two-generator horizontal forms on selected coordinate planes, not the
full flat development form.  These forms are introduced in
Section~\ref{sec:simplices-stokes}; their first-layer geometric
interpretation is summarized in Section~\ref{sec:geometry-first-layers},
and the Heisenberg and Engel Stokes calculations are carried out in
Section~\ref{sec:stokes-heis-engel}.  The corresponding independent
third-layer directions in the free model are discussed in
Section~\ref{sec:free-three-step}.  

When \(G\) is a vector group, BCH becomes ordinary addition and
\eqref{eq:exact-pi1-derham} reduces to the classical period map
\([c]\mapsto\int_c\alpha\).  Thus, Theorem~\ref{thm:nilpotent-pi1-derham} is the nonabelian period statement appropriate to this paper: the
Lie integral supplies logarithmic coordinates for multiplicative
holonomy.

The asymptotic-development theorem, Theorem~\ref{thm:macroscopic-path},
separates the leading path geometry along a convergent scale sequence
from the effect of closing an orbit segment.  In a nonabelian development, history matters: two horizontal
paths with the same endpoint in \(V_1\) can have different developed
endpoints in \(G\), because the higher layers retain information
about the chronological order of their increments.  Bounded closings
vanish after unit-interval and Carnot rescaling, but they can enter
subleading terms through the BCH law.  In the rest of this section, we use the finite-time holonomy statement, the asymptotic-development theorem, and bounded-closing invariance to produce a Carnot-scale nilpotent \(\pi_1\)-de Rham theorem.

\subsection{Finite-time homotopy data and the role of a filling}
\label{sec:invariance}

Let \(c:[0,1]\to M\) be a loop based at \(x_0\).  For the flat right
Maurer--Cartan development form of Section~\ref{sec:nilpotent-development}, its developed
endpoint is the holonomy element
\[
  \exp\!\left(\Lint_c\mathcal A\right)=\rho([c]),
  \qquad
  \Lint_c\mathcal A=\log\rho([c]).
\]
If \(c_0\) and \(c_1\) are homotopic through loops based at \(x_0\),
then \([c_0]=[c_1]\) in \(\pi_1(M,x_0)\), so their developed endpoints
in \(G\) coincide.  Equivalently, since \(\exp:\g\to G\) is a global
diffeomorphism, their Lie integrals coincide.  This is the nilpotent \(\pi_1\)-de Rham statement and does not require choosing a
spanning surface.

The auxiliary two-generator horizontal connection forms
\[
  \mathcal A_{ij}=-\omega_{ij}
\]
introduced in Section~\ref{sec:simplices-stokes} for the Heisenberg and
Engel calculations are not the full flat Maurer--Cartan development
form.  Their nonzero curvatures allow us to express the same boundary
logarithm geometrically in terms of signed area or a corrected first moment of
signed area.  The nilpotent specialization of the Chacon--Fomenko
nonabelian Stokes theorem, equation~\eqref{eq:CF-Stokes-pair}, identifies
the Lie integral of the complete oriented boundary with the surface
Lie integral of the twisted curvature over the chosen parameterized
filling.

The asymptotic issue is different: an open orbit segment must be closed
before it represents an element of \(\pi_1\).  A bounded closing can alter
subleading logarithmic terms through BCH, but it disappears from the
leading Carnot-scaled limit.  The next proposition makes this precise.

\subsection{Invariance under bounded closings}
\label{sec:perturbation}

\begin{proposition}[Invariance of Carnot-scale limits under bounded factors]
\label{prop:perturbation}
Let \(G\) be a simply connected graded \(n\)-step nilpotent group and
write
\[
  g(T)=\exp H[T],\qquad H[T]=\sum_{k=1}^nH_k[T],
  \qquad H_k[T]=O(T^k).
\]
Let
\[
  b_i(T)=\exp Q_i(T),\qquad Q_i(T)=\sum_{k=1}^nQ_{i,k}(T),
  \qquad i=1,2,
\]
where \(Q_1(T)\) and \(Q_2(T)\) are uniformly bounded, and let
\[
  \widetilde g(T)=b_1(T)\,g(T)\,b_2(T).
\]
With the Chacon--Fomenko chronological convention, a factor appended
after the segment multiplies on the left and one preceding it on the
right.  Then
\[
  \widetilde H[T]=\log\widetilde g(T)
\]
satisfies
\[
  \widetilde H_k[T]-H_k[T]=O(T^{k-1}),
  \qquad 1\leq k\leq n.
\]
Consequently,
\[
  \delta_{1/T}\widetilde H[T]
  -\delta_{1/T}H[T]
  \longrightarrow0
  \qquad\text{as }T\to\infty.
\]
\end{proposition}

\begin{proof}
At the group level, the reason for the invariance is immediate.  Since
\(Q_1(T)\) and \(Q_2(T)\) are bounded,
\[
  \delta_{1/T}Q_i(T)\longrightarrow0,
  \qquad
  \delta_{1/T}b_i(T)\longrightarrow e.
\]
Because the Carnot dilations are group automorphisms,
\[
  \delta_{1/T}\widetilde g(T)
  =\delta_{1/T}b_1(T)\,\delta_{1/T}g(T)\,\delta_{1/T}b_2(T),
\]
so both bounded factors disappear at the leading Carnot scale.

The BCH expansion gives the sharper layerwise estimate.  Decompose each
BCH Lie monomial into its homogeneous graded components.  Apart from
\(H[T]\) itself, every monomial contributing to
\[
  \widetilde H[T]-H[T]
\]
contains at least one homogeneous component of \(Q_1(T)\) or
\(Q_2(T)\).  Suppose that the total graded degree contributed by all of
its \(Q_1(T)\)- and \(Q_2(T)\)-factors is
\(q\geq1\).  Since a monomial contributing to \(V_k\) has total graded
degree \(k\), its \(H[T]\)-factors have total graded degree \(k-q\).
Each component of \(Q_1(T)\) and \(Q_2(T)\) is bounded, while
\(H_j[T]=O(T^j)\); so, by 
multilinearity and continuity of the Lie bracket, such a monomial is \(O(T^{k-q})\), and since \(q\geq 1\), it is also \(O(T^{k-1}).\)   The iterated BCH expansion is finite, so summing its terms yields
\[
  \widetilde H_k[T]-H_k[T]=O(T^{k-1}).
\]
After applying \(\delta_{1/T}\), the difference in each graded component
is \(O(T^{-1})\), and hence tends to zero.
\end{proof}

\begin{remark}[Subsequential limits]
The proof gives more than preservation of an already existing limit.  For
every sequence \(T_j\to\infty\),
\[
  \delta_{1/T_j}\widetilde H[T_j]
  -\delta_{1/T_j}H[T_j]
  \longrightarrow0.
\]
Thus, the two normalized families have the same convergence behaviour
along every sequence and exactly the same set of subsequential limits.
There is no reason for this common set to consist of a single point:
different macroscopic subsequences may have different Carnot-scale
developments.  At the horizontal level, Proposition~\ref{prop:subsequential-realization} shows that, across the class of bounded inputs, there is no restriction on subsequential macroscopic shape beyond being Lipschitz and based at the origin.
\end{remark}

\begin{remark}[Subleading dependence on the closing]
Proposition~\ref{prop:perturbation} concerns only the leading Carnot-scale
behaviour.  A bounded closing changes the \(V_k\)-component of the
logarithm by \(O(T^{k-1})\).  This change disappears after division by
\(T^k\), but it may contribute at the next asymptotic order.
Section~\ref{sec:path-defects} organizes the successive asymptotic coefficients of each homogeneous layer along a chosen scale sequence into the \emph{Chacon--Fomenko period array}.  When these coefficients exist,
a bounded closing may change subleading entries of that array
even though it leaves the Carnot diagonal unchanged.
\end{remark}

\subsection{The Carnot-Scale Nilpotent \texorpdfstring{$\pi_1$}{pi1} de Rham Theorem}
\label{sec:carnot-pi1-derham}

The nilpotent \(\pi_1\)-de Rham correspondence and the
asymptotic-development theorem answer different questions.  The first assigns
to each finite based loop its nilpotent holonomy.  The second
identifies the Carnot-scale limit of an open horizontal development.  An
open orbit segment does not itself determine an element of \(\pi_1\), but
a bounded closing turns it into a based loop.  Proposition~\ref{prop:perturbation}
shows that the closing changes only lower asymptotic orders.  Thus, the three
statements fit together at the leading Carnot scale.

The compatibility hypothesis in the next theorem is essential.  The
flat form \(\mathcal A\) determines the finite-time holonomy, while
the asymptotic-development theorem is formulated for a horizontal BV path in
\(V_1\).  We assume that, along the long path under consideration, these
two descriptions give the same finite-time developed endpoint.  As
explained in Section~\ref{sec:full-versus-horizontal}, this is not automatic
for an arbitrary prescribed horizontal input.

\begin{theorem}[Carnot-scale nilpotent \(\pi_1\)-de Rham theorem]
\label{thm:carnot-pi1-derham}
Let \(M\) be a connected smooth manifold with base point \(x_0\), let
\[
  \g=V_1\oplus\cdots\oplus V_n
\]
be a finite-dimensional graded \(n\)-step nilpotent Lie algebra generated
by \(V_1\), and let \(G=\exp\g\).  Let
\[
  \rho:\pi_1(M,x_0)\longrightarrow G
\]
be a representation, and let \(\mathcal A\) be a flat development form in
the corresponding based gauge class.

Let \(c:[0,\infty)\to M\) be a path with \(c(0)=x_0\), absolutely
continuous on compact intervals.  Suppose there is a continuous path
\[
  \ell:[0,\infty)\longrightarrow V_1
\]
of bounded variation on compact intervals such that, for every \(T>0\),
its development agrees with the flat development of the initial
segment \(c_T=c|_{[0,T]}\):
\begin{equation}
  H[T]
  :=H[\ell|_{[0,T]}]
  =\Lint_{c_T}\mathcal A.
  \label{eq:carnot-pi1-compatible-development}
\end{equation}
For each \(T\), choose an absolutely continuous closing path \(b_T\) from
\(c(T)\) to \(x_0\), set
\[
  \gamma_T=b_T*c_T,
\]
and assume that the logarithmic developed endpoints of the closings,
\[
  Q(T):=\Lint_{b_T}\mathcal A,
\]
remain bounded in \(\g\).

Define
\[
  \ell_T(u)=\frac{\ell(Tu)-\ell(0)}{T},
  \qquad 0\leq u\leq1.
\]
Let \(T_j\to\infty\) be a sequence.  Assume that
\(\ell_{T_j}\to\ell^\infty\) uniformly on \([0,1]\), where
\(\ell^\infty\) is a continuous BV path, and that
\[
  \sup_j\Var(\ell_{T_j};[0,1])<\infty.
\]
Then
\begin{equation}
  \delta_{1/T_j}\log\rho([\gamma_{T_j}])
  \longrightarrow
  H[\ell^\infty]
  \qquad\text{in }\g,
  \label{eq:carnot-pi1-algebra-limit}
\end{equation}
and, equivalently,
\begin{equation}
  \delta_{1/T_j}\rho([\gamma_{T_j}])
  \longrightarrow
  \exp H[\ell^\infty]
  \qquad\text{in }G.
  \label{eq:carnot-pi1-group-limit}
\end{equation}
These sequential limits are independent of the family of closing paths
among all families for which \(Q(T)\) remains bounded.  If the convergence
and variation hypotheses hold for the full family \(\ell_T\), then the
same conclusions hold as full limits as \(T\to\infty\).
\end{theorem}

\begin{proof}
By the endpoint identity~\eqref{eq:exact-endpoint} and the chronological concatenation
convention,
\[
  \rho([\gamma_T])
  =\exp Q(T)\exp H[T].
\]
Hence, with
\[
  \widetilde H[T]:=\log\rho([\gamma_T]),
\]
we have
\[
  \widetilde H[T]=\operatorname{BCH}(Q(T),H[T]).
\]
The hypotheses on \(\ell_{T_j}\) and Theorem~\ref{thm:macroscopic-path} give
\[
  \delta_{1/T_j}H[T_j]\longrightarrow H[\ell^\infty]
\]
and, componentwise, \(H_k[T_j]=O(T_j^k)\).  Since \(Q(T_j)\) is bounded,
Proposition~\ref{prop:perturbation}, applied to the restricted family, gives
\[
  \delta_{1/T_j}\widetilde H[T_j]
  -\delta_{1/T_j}H[T_j]
  \longrightarrow0.
\]
This proves \eqref{eq:carnot-pi1-algebra-limit}.  Exponentiating and using
\(\delta_r(\exp X)=\exp(\delta_rX)\) gives
\eqref{eq:carnot-pi1-group-limit}.  The same argument applies to every
family of closings with bounded logarithmic developed endpoints, so the
limit is independent of that choice.
\end{proof}

\begin{corollary}[Straight macroscopic horizontal path]
\label{cor:carnot-pi1-straight}
Under the hypotheses of Theorem~\ref{thm:carnot-pi1-derham}, suppose
that
\[
  \ell^\infty(u)=u\mu
\]
for some \(\mu\in V_1\).  Write
\[
  \log\rho([\gamma_T])
  =\widetilde H_1[T]+\cdots+\widetilde H_n[T],
  \qquad \widetilde H_k[T]\in V_k.
\]
Then
\[
  \frac{\widetilde H_1[T_j]}{T_j}\longrightarrow\mu,
  \qquad
  \frac{\widetilde H_k[T_j]}{T_j^k}\longrightarrow0,
  \quad 2\leq k\leq n,
\]
and
\[
  \delta_{1/T_j}\rho([\gamma_{T_j}])\longrightarrow\exp\mu.
\]
\end{corollary}

\begin{proof}
The Lie integral of the straight path \(u\mapsto u\mu\) is \(\mu\)
because every bracket of its increments vanishes.  The conclusion follows
from Theorem~\ref{thm:carnot-pi1-derham}.
\end{proof}

\begin{remark}[Geometric bounded closings]
The Carnot-scale nilpotent \(\pi_1\)-de Rham theorem is formulated directly in terms of the developed closing
quantity \(Q(T)\) for the chosen flat form.  If \(M\) is compact, \(\mathcal A\) is smooth, and the
closings have uniformly bounded length with respect to any fixed Riemannian metric on \(M\), then the corresponding path inputs have uniformly
bounded total variation.  The finite nilpotent Lie-integral formulas then imply
that \(Q(T)\) remains bounded.
\end{remark}

\begin{example}[The natural class: horizontal trajectories on a compact nilmanifold]
\label{ex:natural-class}
Let $\Gamma\subset G$ be a lattice, let $M=G/\Gamma$ with base point
$x_0=e\Gamma$, and identify the universal covering with
$p:G\to G/\Gamma$, on which $\pi_1(M,x_0)\cong\Gamma$ acts by right
translation.  The identity map $D=\operatorname{id}_G$ is a normalized
developing map: it is equivariant for the inclusion
$\rho:\Gamma\hookrightarrow G$, and $D^*\theta^R=\theta^R$ is
right-invariant, hence descends to a flat form
$\mathcal A\in\Omega^1(M;\g)$.  This is the natural instance of
Theorem~\ref{thm:nilpotent-pi1-derham}.

Call a path $c:[0,\infty)\to M$ with $c(0)=x_0$, absolutely continuous
on compact intervals, \emph{$\mathcal A$-horizontal} if
\[
  K(t):=\mathcal A_{c(t)}(\dot c(t))\in V_1
  \qquad\text{for almost every }t;
\]
equivalently, the lifted trajectory satisfies
$\widetilde c\,'(t)\,\widetilde c(t)^{-1}\in V_1$ almost everywhere.
Put $\ell(t)=\int_0^tK(s)\,ds$.  Then $K$ is locally integrable,
$\ell$ is continuous and of bounded variation on compact intervals,
and, because $K$ takes values in $V_1$, the flat development of $c_T$
and the Chacon--Fomenko Lie integral of the horizontal primitive are
the same Lie integral of the same input:
\[
  \Lint_{c_T}\mathcal A
  =\Lint_0^TK(t)\,dt
  =H\bigl[\ell|_{[0,T]}\bigr].
\]
Thus, the compatibility hypothesis
\eqref{eq:carnot-pi1-compatible-development} holds by
construction.  Since $M$ is compact, closing paths of uniformly
bounded Riemannian length exist for every $T$, so the preceding remark
gives $Q(T)$ bounded.  Hence,
Theorem~\ref{thm:carnot-pi1-derham} applies to every
$\mathcal A$-horizontal trajectory on a compact nilmanifold, along
every scale sequence for which the rescaled horizontal primitives
converge with uniformly bounded variation.
\end{example}

\begin{corollary}[Topological realization on nilmanifolds]
\label{cor:nilmanifold-realization}
Let $\eta:[0,1]\to V_1$ be Lipschitz with $\eta(0)=0$.  Then there
exist an $\mathcal A$-horizontal trajectory $c$ on $M=G/\Gamma$ and a
sequence $T_j\to\infty$ such that, for every family of closings with
$Q(T)$ bounded,
\[
  \delta_{1/T_j}\log\rho([\gamma_{T_j}])
  \longrightarrow H[\eta].
\]
\end{corollary}

\begin{proof}
Proposition~\ref{prop:subsequential-realization} provides a bounded
measurable $K$ and a sequence $T_j\to\infty$ for which $\eta$ is the
macroscopic horizontal path of $\ell(t)=\int_0^tK(s)\,ds$.  Let
$g:[0,\infty)\to G$ solve $g'(t)=K(t)g(t)$, $g(0)=e$, and let
$c=p\circ g$.  Then $c$ is $\mathcal A$-horizontal with horizontal
primitive $\ell$, and the conclusion follows from
Example~\ref{ex:natural-class} and
Theorem~\ref{thm:carnot-pi1-derham}.
\end{proof}

\begin{remark}[Regularity and the smooth category]
\label{rem:smooth-category}
The locally absolutely continuous class is the smallest one closed
under the constructions of this paper: the developed trajectory of a
locally integrable horizontal input is locally absolutely continuous,
and Proposition~\ref{prop:subsequential-realization} produces bounded
measurable inputs.  All piecewise $C^1$, and in particular all smooth,
trajectories are included.  Conversely, at the leading Carnot scale
nothing is lost by remaining in the smooth category: since $\ell$ is
$\lVert K\rVert_\infty$-Lipschitz, a fixed smoothing of a bounded
input changes the rescaled paths $\ell_T$ by $O(T^{-1})$ uniformly on
$[0,1]$ and generates the same macroscopic horizontal paths along
every scale sequence.  In particular, the trajectory in
Corollary~\ref{cor:nilmanifold-realization} may be taken $C^\infty$.
Subleading period-array entries, by contrast, are not generally
preserved under smoothing.
\end{remark}

\begin{example}[Horizontal flows on compact manifolds]
\label{ex:horizontal-flows}
Let $M$ be compact and connected with base point $x_0$, let
$\rho:\pi_1(M,x_0)\to G$ be a representation with flat development
form $\mathcal A$ as in Theorem~\ref{thm:nilpotent-pi1-derham}, and
let $\phi^t$ be a smooth flow on $M$ generated by a vector field $V$.
Suppose that the flow is \emph{$\mathcal A$-horizontal}:
\[
  F:=\mathcal A(V):M\longrightarrow\g
  \qquad\text{takes values in }V_1.
\]
Let $c(t)=\phi^t(x_0)$ be the orbit of the base point.  Then
\[
  K(t)=\mathcal A_{c(t)}(\dot c(t))=F(\phi^tx_0)\in V_1,
\]
so, with $\ell(t)=\int_0^tK(s)\,ds$, the compatibility hypothesis
\eqref{eq:carnot-pi1-compatible-development} holds by construction,
exactly as in Example~\ref{ex:natural-class}.  Since $M$ is
compact, $F$ is bounded; hence
$\Var(\ell_T;[0,1])\leq\lVert F\rVert_\infty$ for every $T$, and the
uniform variation hypothesis of
Theorem~\ref{thm:carnot-pi1-derham} is automatic.  Closings of
uniformly bounded Riemannian length exist by compactness and give
$Q(T)$ bounded.  So, Theorem~\ref{thm:carnot-pi1-derham}
applies to the orbit along every scale sequence for which the
rescaled horizontal primitives converge uniformly.  Flows on compact
manifolds are the setting of \cite{Schw57} and \cite{BM93}; for an
orbit through a point other than $x_0$, a fixed connecting arc
contributes a further bounded factor, which
Proposition~\ref{prop:perturbation} absorbs.

Suppose in addition that $\phi^t$ is uniquely ergodic with invariant
probability measure $m$.  Then the Birkhoff averages
$\frac1T\int_0^TF(\phi^sx)\,ds$ converge to
$\mu=\int_MF\,dm$, uniformly in $x$.
Corollary~\ref{cor:bounded-mean-straight} places every orbit in the
full-limit regime, and Corollary~\ref{cor:carnot-pi1-straight} gives
\[
  \delta_{1/T}\rho([\gamma_T])\longrightarrow\exp\mu
\]
as a full limit, independent of the orbit, of the bounded closings,
and of the scale sequence.  At the leading Carnot scale, the
asymptotic homotopy of a uniquely ergodic horizontal flow is the single group element $\exp\mu$; the finer nonabelian
information lies in the subleading defects of
Section~\ref{sec:path-defects}.

The horizontal nilflows $\phi^t(g\Gamma)=\exp(tX)g\Gamma$ with
$X\in V_1$ on a compact nilmanifold $G/\Gamma$ belong both to this
class and to Example~\ref{ex:natural-class}, with constant
input $K\equiv X$.
\end{example}

\begin{remark}[Scope of the theorem]
Theorem~\ref{thm:carnot-pi1-derham} is a leading Carnot-scale theorem.
It identifies, along every convergent scale sequence, the asymptotic-cone
endpoint of the holonomy elements associated with long paths completed by
bounded closings.  It does not assert that the subleading
period coefficients exist, nor that those which do exist are independent
of bounded closings.  The depth-two Benardete--Mitchell correction is the
first instance in which a lower asymptotic order acquires a
closing-independent homotopy interpretation.  The remaining gap between
this Carnot-scale theorem and a full nilpotent \(\pi_1\)-de Rham theorem is
summarized in Section~\ref{sec:conclusion}.
\end{remark}

\section{Period arrays, path defects, and asymptotic homotopy}
\label{sec:path-defects}

The asymptotic-development theorem identifies what survives at the leading Carnot scale along any convergent scale sequence.  The new realization theorem shows that this leading object need not be straight: every Lipschitz horizontal path based at the origin can occur along a subsequence of a bounded input.  Accordingly, the leading diagonal data may retain signed area and higher area-moment information.  Once a scale sequence has been fixed, one may ask for successive lower-order terms in the complete homogeneous components.  The full-limit regime is the rigid special case in which the macroscopic path is straight and every higher diagonal entry vanishes.

\subsection{The period array}

For a bounded horizontal input \(K\), with \(\ell(t)=\int_0^tK(s)\,ds\), the exact scaling identity \eqref{eq:exact-path-scaling-components} and the estimate
\[
  \Var(\ell_T;[0,1])\leq \|K\|_\infty
\]
imply \(H_k[T]=O(T^k)\).  This growth bound does not imply an asymptotic expansion.  When successive coefficients exist along a chosen scale sequence, they must be extracted from the complete homogeneous component, after all simplex terms and recursive interaction brackets have been combined.

\begin{definition}[Sequential period-admissibility and the period array]
\label{def:CF-power-expansion}
Let \(K:[0,\infty)\to V_1\) be bounded and locally integrable, let \(H_k[T]\in V_k\) be the depth-\(k\) component of its Lie integral, and fix a scale sequence
\[
  \mathbf T=(T_j),
  \qquad T_j\longrightarrow\infty.
\]
For a fixed layer \(k\), define its coefficients successively, whenever the indicated limits exist, by
\begin{equation}
  \Lamn{k}{k}(\mathbf T)
  =\lim_{j\to\infty}\frac{H_k[T_j]}{T_j^k},
  \label{eq:sequential-period-diagonal}
\end{equation}
and, for \(1\leq m<k\), by
\begin{equation}
  \Lamn{k}{m}(\mathbf T)
  =\lim_{j\to\infty}
  \frac{
    H_k[T_j]-\displaystyle\sum_{q=m+1}^{k}
      T_j^q\Lamn{k}{q}(\mathbf T)
  }{T_j^m}.
  \label{eq:sequential-period-recursion}
\end{equation}
We say that \(K\) is \emph{period-admissible through step \(r\) along \(\mathbf T\)} if all these limits exist for \(1\leq m\leq k\leq r\).  Equivalently, for every \(1\leq k\leq r\),
\begin{equation}
  H_k[T_j]
  =\sum_{m=1}^{k}T_j^m\Lamn{k}{m}(\mathbf T)+o(T_j)
  \qquad (j\to\infty).
  \label{eq:CF-power-expansion-general}
\end{equation}
The triangular family
\[
  \widehat{\Lambda}^{\infty}_{\mathbf T}(K)
  =\{\Lamn{k}{m}(\mathbf T):1\leq m\leq k\leq r\}
\]
is the \emph{sequential Chacon--Fomenko period array through step \(r\) along \(\mathbf T\)}.  The subscript on \(\Lambda\) records the homogeneous layer and the parenthesized superscript records the power of \(T_j\).  If the recursive limits exist as \(T\to\infty\) without restriction to a sequence, then the same array is obtained along every scale sequence, and this is the \emph{full period array}.
\end{definition}

Homogeneity allows us to rewrite \eqref{eq:CF-power-expansion-general} as
\begin{equation}
  H_k[\ell_{T_j}]
  =\Lamn{k}{k}(\mathbf T)
   +\frac{\Lamn{k}{k-1}(\mathbf T)}{T_j}+\cdots+
    \frac{\Lamn{k}{1}(\mathbf T)}{T_j^{k-1}}+o(T_j^{1-k}).
  \label{eq:normalized-period-expansion}
\end{equation}
The diagonal is the leading Carnot-normalized limit along the chosen sequence, while the lower entries are successive defects of the functional \(H_k[\ell_{T_j}]\).  They need not arise from an expansion of the paths \(\ell_{T_j}\) themselves.

\begin{proposition}[Sequential macroscopic determination of the diagonal]
\label{prop:macro-cancellation-criterion}
Let \(\mathbf T=(T_j)\) be a scale sequence such that
\[
  \ell_{T_j}\longrightarrow\ell^\infty
\]
uniformly on \([0,1]\), with
\[
  \sup_j\Var(\ell_{T_j};[0,1])<\infty.
\]
Then, for every \(1\leq k\leq n\), the diagonal limit in \eqref{eq:sequential-period-diagonal} exists and
\[
  \Lamn{k}{k}(\mathbf T)=H_k[\ell^\infty].
\]
Consequently, whenever the input is period-admissible through step \(r\) along \(\mathbf T\), the diagonal of its sequential period array is exactly the homogeneous logarithmic development of the macroscopic path arising along that sequence.
\end{proposition}

\begin{proof}
By exact homogeneity,
\[
  \frac{H_k[T_j]}{T_j^k}=H_k[\ell_{T_j}].
\]
Theorem~\ref{thm:macroscopic-path} gives
\[
  H_k[\ell_{T_j}]\longrightarrow H_k[\ell^\infty],
\]
which proves both the existence and the stated value of the diagonal limit.
\end{proof}

\begin{corollary}[Full macroscopic-limit rigidity of the diagonal]
\label{cor:full-limit-period-diagonal}
Suppose that the full rescaled family \(\ell_T\) converges uniformly as \(T\to\infty\), with uniformly bounded variation.  Then there is \(\mu\in V_1\) such that
\[
  \ell^\infty(u)=u\mu,
\]
and, for every scale sequence \(\mathbf T=(T_j)\),
\[
  \Lamn{1}{1}(\mathbf T)=\mu,
  \qquad
  \Lamn{k}{k}(\mathbf T)=0,
  \quad 2\leq k\leq n.
\]
In particular, this conclusion holds for every bounded input with an ordinary mean \(\mu\).  Thus, whenever a sequential period array exists in the full-limit regime, its higher diagonal is trivial and independent of the chosen scale sequence.
\end{corollary}

\begin{proof}
Proposition~\ref{prop:full-macroscopic-rigidity} makes the full macroscopic limit straight.  Corollary~\ref{cor:straight-line} gives
\[
  H_1[\ell^\infty]=\mu,
  \qquad
  H_k[\ell^\infty]=0\quad(k\geq2),
\]
and Proposition~\ref{prop:macro-cancellation-criterion} identifies these values with the diagonal limits along every sequence.  The ordinary-mean assertion follows from Corollary~\ref{cor:bounded-mean-straight}.
\end{proof}

The coefficients in each row are determined successively along the chosen sequence.  After \(\Lamn{k}{k}(\mathbf T)\) has been found, the next coefficient is the limit of the normalized remainder, and the procedure continues down to the linear scale.  This formulation permits partial rows: one may know the diagonal and the next term without knowing whether the following normalized remainder converges.  Different scale sequences may yield different arrays, beginning with different diagonals when their macroscopic paths differ.  Fractional powers, logarithmic terms, or persistent oscillations may prevent period-admissibility along a given sequence even though its leading Carnot limit exists.

Suppressing \(\mathbf T\) from the notation when the sequence is fixed, the array through step three has the form
\begin{align}
  H_1[T_j]&=T_j\Lamn{1}{1}+o(T_j),\notag\\
  H_2[T_j]&=T_j^2\Lamn{2}{2}+T_j\Lamn{2}{1}+o(T_j),\notag\\
  H_3[T_j]&=T_j^3\Lamn{3}{3}+T_j^2\Lamn{3}{2}+T_j\Lamn{3}{1}+o(T_j).
  \label{eq:three-step-period-array}
\end{align}
For a nonstraight subsequential macroscopic path, the higher diagonal terms may be nonzero and are determined by its signed-area and area-moment development.  In the full-limit regime, and in particular for a bounded input with an ordinary mean, the higher diagonal terms vanish.  The first nonabelian subleading data then occur in \(\Lamn{2}{1}\), \(\Lamn{3}{2}\), and \(\Lamn{3}{1}\).  Their existence is a dynamical and analytic question distinct from the finite nilpotent algebra that defines them.

\subsection{The Benardete--Mitchell correction in depth two}

\label{sec:BM-correction}

In depth two, a bounded closing can change the central coefficient at
order \(t\), but the BCH law produces only one central interaction to be
corrected.  Benardete and Mitchell \cite{BM93} use this correction to construct a
closing-independent asymptotic homotopy invariant for 2-step nilpotent Lie groups.  In depth three and
higher, additional brackets involving bounded endpoint defects occur at
several subleading orders, so the same argument does not extend without
further corrections.

\subsubsection*{The depth-two pullback group}
Following \cite{BM93}, let \(N\) be a torsion-free 2-step nilpotent quotient of
\(\pi_1(M,x_0)\), and let \(L(N)\) be the Lie algebra of its real
Mal'cev completion.  The Benardete--Mitchell pullback group
\(\widehat N\) is the vector space \(L(N)\) equipped with the truncated
BCH product
\[
  g\times h=g+h+\frac12[g,h].
\]
The exponential map is the identity in these coordinates, so an element
of \(\widehat N\) and its logarithm are represented by the same vector.
For an open trajectory, we may regard the Lie integral
\[
  H[t]=H_1[t]+H_2[t]
\]
simultaneously as the logarithm of its development and as the
corresponding element of \(\widehat N\).  Assume that the abelian
asymptotic limit
\[
  \mu=\lim_{t\to\infty}\frac{H_1[t]}{t}
\]
exists.  Benardete and Mitchell define their asymptotic limit by the
corrected expression
\begin{equation}
  \operatorname{Alim}(H[t])
  =\lim_{t\to\infty}
    \left(\frac{H[t]}{t}+\frac12[H[t],\mu]\right),
  \label{eq:BM-Alim}
\end{equation}
whenever this limit exists~\cite[Lemma~4.2]{BM93}.  Since \(V_2\) is
central, the correction depends only on \(H_1[t]\), and the formula
separates into its horizontal and central parts as follows.

\begin{proposition}[Reconstruction of the Benardete--Mitchell Alim invariant]
\label{prop:bm-recovery}
Let
\[
  H[t]=H_1[t]+H_2[t]
\]
be the Lie integral, equivalently the logarithm of the development,
in a 2-step nilpotent Lie algebra, and let
\(\mu=\lim_{t\to\infty}H_1[t]/t\) be its abelian asymptotic limit.  Then
the Benardete--Mitchell Alim invariant is reconstructed by
\begin{equation}
  \operatorname{Alim}(H[t])
  =\mu+
   \lim_{t\to\infty}
   \left(\frac{H_2[t]}{t}+\frac12[H_1[t]-t\mu,\mu]\right),
  \label{eq:BM-from-CF}
\end{equation}
whenever the displayed limit exists.

In the Heisenberg notation of Section~\ref{sec:Heisenberg}, with
\(H_2[t]=\frac12A(t)Z\),
\(F(t)=\bar f t+\widetilde F(t)\), and
\(G(t)=\bar g t+\widetilde G(t)\), the central coefficient is
\begin{equation}
  \Lambda^{(1)}_{2,\mathrm{BM}}\big|_Z
  =\frac12\lim_{t\to\infty}\frac{A_{\mathrm{BM}}(t)}{t},
  \qquad
  A_{\mathrm{BM}}(t)
  =A(t)+\bar g\,t\widetilde F(t)-\bar f\,t\widetilde G(t).
  \label{eq:BM-corrected-area-CF}
\end{equation}
\end{proposition}

\begin{proof}
Equation~\eqref{eq:BM-Alim} and the centrality of \(V_2\) give
\([H_2[t],\mu]=0\), and hence \eqref{eq:BM-from-CF}.  In the Heisenberg
basis,
\[
  [H_1[t],\mu]
  =\bigl(\bar g\widetilde F(t)-\bar f\widetilde G(t)\bigr)Z,
\]
which gives \eqref{eq:BM-corrected-area-CF}.  The plus sign is the sign in
the Benardete--Mitchell formula; the comparison with probabilistic L\'evy area is a
separate orientation convention.
\end{proof}

We distinguish the corrected coefficient \(\Lambda_{2,\mathrm{BM}}^{(1)}\)
from the raw period-array entry \(\Lambda_2^{(1)}\): the former includes
the nonlinear correction determined by the horizontal mean.

The coefficient \(H_2[t]/t\) need not be independent of the bounded
closing.  If a closing contributes a bounded horizontal element
\(c_t\in V_1\), its BCH product with the accumulated horizontal term
\(H_1[t]\sim t\mu\) produces a central contribution proportional to
\([c_t,H_1[t]]\), which can survive after division by \(t\).  The
correction in \eqref{eq:BM-from-CF} compensates for precisely this
interaction.  Benardete and Mitchell prove that \(\operatorname{Alim}\)
is unchanged by bounded right perturbations in the 2-step pullback group
\cite[Proposition~4.3]{BM93}, and they apply this fact to the homotopy
images of loops obtained by adding bounded closing paths
\cite[Proposition~5.2]{BM93}.  Thus, the corrected limit is
independent of the chosen bounded closing and defines their asymptotic
homotopy invariant.  In the Chacon--Fomenko chronological convention, an appended
closing factor multiplies on the left rather than on the right; this
changes the corresponding bracket order, but not the underlying
bounded-closing mechanism after the conventions are translated.

\subsubsection*{The higher-step obstruction}
The depth-two argument is special because every BCH correction lands in
the centre and no nested bracket survives.  In a 3-step Lie algebra, a
bounded correction in a lower layer can bracket with an accumulated term
and contribute at a subleading order in a higher layer.  For example,
\[
  c_t\in V_1,\qquad H_2[t]=O(t)
\]
can produce \([c_t,H_2[t]]\in V_3\) of order \(t\).  Likewise, a bounded
\(V_2\)-correction can bracket with the linear horizontal drift
\(H_1[t]\sim t\mu\) and again produce an order-\(t\) contribution in
\(V_3\).  These terms reflect the nested closing action rather than a
coordinate error.

These are the first nested closing terms responsible for the failure of
the unmodified Benardete--Mitchell bounded-perturbation argument beyond depth two.  The
Chacon--Fomenko expansion identifies the affected subleading coefficients, but it
does not make them closing-independent; absent further corrections or a
suitable quotient, they remain data of the chosen finite-time
development.

\subsubsection*{The two normalizations}
By Proposition~\ref{prop:perturbation}, these bounded-closing effects
disappear at the leading Carnot scale.  The distinction between the
Carnot and Benardete--Mitchell normalizations is substantive.  The
Carnot dilation \(\delta_{1/t}\) scales the depth-\(k\) stratum by
\(t^{-k}\) and reads the leading coefficient \(\Lamn{k}{k}\); it records
what the orbit looks like under graded rescaling.  The Benardete--Mitchell normalization
uses a uniform factor \(1/t\) in every depth because it compares the
orbit with a one-parameter subgroup \(t\mapsto a^t\) generated by an
element \(a\in\widehat N\).  In the pullback group \(\widehat N\), the
identity \(a^t=ta\) in these logarithmic coordinates makes this linear
normalization natural.  Thus, the two normalizations extract
different asymptotic orders from the same layers.

The Carnot diagonal can be too coarse: the higher diagonal
terms may vanish while nonabelian information survives at
subleading orders.  At the first layer, Schwartzman's asymptotic cycle
records \(\Lamn{1}{1}\); in depth two, the Benardete--Mitchell construction
supplements it with the Alim-corrected central coefficient.  In depth three
and higher, the lower asymptotic orders are natural data of the chosen
path-holonomy input, but they are not generally homotopy invariants and
need not be invariant under topological conjugacy without compatible
development data.  This distinction is one of the main points of the
present paper.

\subsection{The free 3-step obstruction}
\label{sec:free-three-step}

The Engel Lie group shows that the third homogeneous component is a corrected first moment of area, but its third layer is one-dimensional.  It cannot separate a repeated-generator area moment from a contribution involving three distinct horizontal generators.  The free 3-step Lie algebra on three generators is the first model in which these mechanisms occupy independent directions.

Let \(\mathfrak f_{3,3}\) be generated by \(X_1,X_2,X_3\).  Put
\[
  Y_1=[X_1,X_2],\qquad
  Y_2=[X_1,X_3],\qquad
  Y_3=[X_2,X_3].
\]
A convenient basis of the third layer is
\begin{align*}
  Z_1&=[Y_1,X_1], & Z_2&=[Y_1,X_2], & Z_3&=[Y_1,X_3],\\
  Z_4&=[Y_2,X_1], & Z_5&=[Y_2,X_2], & Z_6&=[Y_2,X_3],\\
  Z_7&=[Y_3,X_2], & Z_8&=[Y_3,X_3],
\end{align*}
with \([Y_3,X_1]=Z_5-Z_3\) by the Jacobi identity.  For
\[
  K(t)=f(t)X_1+g(t)X_2+h(t)X_3,
\]
write \(F_i(t)\) for the three horizontal primitives and \(A_i(t)\) for the signed areas in the coordinate planes.  Then
\[
  H_2[t]=\frac12A_1(t)Y_1+\frac12A_2(t)Y_2+\frac12A_3(t)Y_3,
\]
and the third-layer recursion is
\begin{equation}
  3H_3[t]=T_2[t]+\frac12[H_1[t],H_2[t]].
  \label{eq:F33-depth-three-recursion}
\end{equation}
For each basis vector \(Z_k\in V_3\), write \(H_3^{Z_k}[t]\) for the coefficient of \(Z_k\) in \(H_3[t]\), and write \(B_{Z_k}(t)\) for the coefficient of \(Z_k\) in the triple-simplex term \(T_2[t]\).  Define
\[
  Q_{Z_k}(F_1,F_2,F_3,A_1,A_2,A_3)
\]
to be the coefficient of \(Z_k\) in the \(V_3\)-valued bracket
\[
  [F_1X_1+F_2X_2+F_3X_3,\,
    A_1Y_1+A_2Y_2+A_3Y_3].
\]
Taking the coefficient of \(Z_k\) in \eqref{eq:F33-depth-three-recursion} gives
\begin{equation}
  H_3^{Z_k}[t]
  =\frac13B_{Z_k}(t)
   +\frac1{12}Q_{Z_k}(F_1,F_2,F_3,A_1,A_2,A_3).
  \label{eq:F33-component-form}
\end{equation}

Six third-layer directions arise by allowing one generator of a coordinate pair to act on that pair's area bracket.  They are first moments of the three planar signed areas and span
\[
  V_3^{\mathrm{pair}}
  =\operatorname{span}\{Z_1,Z_2,Z_4,Z_6,Z_7,Z_8\}.
\]
The remaining two directions
\[
  V_3^{\mathrm{cross}}=\operatorname{span}\{Z_3,Z_5\}
\]
have multidegree \((1,1,1)\).  They involve all three horizontal generators and arise from the complete boundary recursion rather than from the twisted curvature of a single coordinate-plane filling.  Thus,
\[
  V_3=V_3^{\mathrm{pair}}\oplus V_3^{\mathrm{cross}}
\]
relative to the chosen ordered horizontal frame.

This decomposition gives a concrete reason that the depth-two closing correction cannot simply be repeated.  A bounded horizontal closing can interact with an order-\(T\) area term, and a bounded second-layer closing can interact with the order-\(T\) horizontal drift.  Both effects contribute at order \(T\) in the third layer, while the all-distinct simplex terms may contribute at the same scale.  A higher-step homotopy invariant must  account simultaneously for same-pair area moments, cross-pair chronological correlations, and nested BCH closing terms.  The present paper identifies these sources but does not construct a universal higher-step correction.

\section{Conclusion}
\label{sec:conclusion}

The endpoint of a path does not capture all of its asymptotic geometry.  A long horizontal path is first rescaled on the unit interval, and its nilpotent development is then compared with the Carnot-rescaled development of the original long segment.  These operations agree exactly at every finite scale.  Under uniform convergence and uniform variation control, the entire developed paths converge.  The limiting endpoint in the asymptotic cone is the endpoint of the development of the limiting macroscopic path.

This formulation separates several structures that are easily merged prematurely.  The horizontal macroscopic path records the shape that survives the blow-down.  Its group-valued development records the chronological noncommutativity of that shape.  The endpoint of the developed path is the corresponding point in the asymptotic cone.  For bounded horizontal inputs, the possible subsequential macroscopic paths are exactly the Lipschitz paths based at the origin, so their higher homogeneous components may retain signed area and higher area-moment information without any further shape restriction.  When the full rescaled family converges, scaling rigidity forces the limiting path to be straight, and the higher leading components vanish.

The passage to asymptotic homotopy occurs only after a long open trajectory has been related to flat nilpotent holonomy and completed by a controlled closing.  The nilpotent \(\pi_1\)-de Rham correspondence identifies the finite-time logarithmic endpoint of a based loop.  Bounded closings vanish at Carnot scale, so the asymptotic-development theorem yields a closing-independent sequential asymptotic-cone class whenever the horizontal and flat developments are compatible.  This is the leading nilpotent asymptotic homotopy statement established here.

Below the leading scale, the problem changes.  Along a chosen scale sequence, the period array begins with the Carnot-scale development of its macroscopic path and then records successive defects in the convergence of the complete homogeneous components.  The lower entries are not determined by the macroscopic path alone, and their existence is not a formal consequence of nilpotent termination.  They may depend on the chosen sequence and may be affected by oscillatory endpoint terms and by bounded closings.  In the full-limit regime, the higher diagonal is trivial, and these subleading defects are the first nonabelian terms in the array.  The Benardete--Mitchell invariant shows that the first depth-two defect admits a nonlinear correction that removes the closing dependence.  The free 3-step calculation displays the additional same-pair, cross-pair, and nested-closing interactions that a higher-step correction must address.

The results distinguish three levels of asymptotic path information.  The first is macroscopic horizontal shape.  The second is the Carnot-scale nilpotent development of that shape.  The third consists of subleading defects that measure how the original family approaches its macroscopic limit.  Only after suitable closing corrections have been constructed do quantities at the third level become candidates for homotopy invariants.

Several problems remain open within this framework.  The asymptotic-development theorem is formulated in a graded target, whereas a general nilpotent Mal'cev completion is filtered and has the associated graded group as its asymptotic cone.  A filtered version should relate the finite-time holonomy to its graded blow-down without making lower-order statements depend on a chosen splitting.  The finite-time topological development also uses a full flat \(\g\)-valued form, while the asymptotic-development theorem is formulated for a horizontal BV input.  A more complete theory should identify geometric hypotheses under which the full development admits the required horizontal realization, or should extend the asymptotic-development theorem to filtered nonhorizontal developments.  Finally, a full higher-step nilpotent \(\pi_1\)-de Rham theorem would require a systematic account of the subleading closing corrections indicated by the free 3-step model.

The path perspective fixes the order in which these questions should be approached.  The asymptotic object is first obtained by developing the macroscopic path.  Topological and dynamical refinements are then asked of that object and of the defects by which the finite-scale developments converge to it.

Conceptually, the viewpoint developed here suggests an interpretation of the Carnot-diagonal data that extends beyond the specific constructions of this paper. The asymptotic cone provides the canonical first-order approximation to the large-scale geometry, while the Carnot-diagonal data and the associated period arrays seek to detect higher-order asymptotic information associated with individual trajectories within that limiting geometry. From this perspective, it is natural to regard the existence of stable nonabelian asymptotic invariants as a finer dynamical phenomenon than the existence of the asymptotic cone itself. The examples considered in this paper support this interpretation, but the present work does not establish it as a general principle. Rather, they suggest that the stabilization of higher-order asymptotic data is itself a mathematical phenomenon worthy of investigation.

This viewpoint also helps explain the distinguished role of the Carnot grading. The graded geometry determines the canonical polynomial asymptotic hierarchy arising from the intrinsic dilation structure of the asymptotic cone, providing the natural setting in which to seek higher-order asymptotic invariants. Other asymptotic models, such as logarithmic or fractional scalings, may well prove fruitful for particular classes of dynamical systems, but they are not dictated by the graded geometry itself. Whether such asymptotic models admit comparably canonical invariants, or whether the Carnot grading marks the natural boundary of a general theory, remains an open question. 

\appendix
\section{Continuity of ordered BV integrals}
\label{app:continuity}

This appendix proves the continuity statement used in Proposition~\ref{prop:pathwise-CF-continuity}.  The estimates are uniform over all subintervals, which allows one to pass from endpoint convergence to convergence of the complete developed paths.

The analytic result needed for the asymptotic-development theorem is the following
continuity statement:
on a family of continuous paths with uniformly bounded variation, each
ordered iterated integral depends continuously on the path in the uniform
norm.

Recall that for a path \(\eta:[a,b]\to\g\),
\begin{equation}
  \Var(\eta;[a,b])
  =
  \sup_{a=t_0<\cdots<t_N=b}
  \sum_{i=1}^N\|\eta(t_i)-\eta(t_{i-1})\|.
  \label{eq:total-variation-recall}
\end{equation}

\begin{lemma}[Uniform limit of uniformly BV paths]
\label{lem:limit-BV}
Let \(\gamma_m:[0,1]\to \g\) be continuous BV paths such that
\[
    \gamma_m\longrightarrow\gamma
\]
uniformly and
\[
    \sup_m\Var(\gamma_m;[0,1])\leq M.
\]
Then \(\gamma\) is a continuous BV path and
\[
    \Var(\gamma;[0,1])\leq M.
\]
\end{lemma}

\begin{proof}
For any partition \(0=t_0<\cdots<t_N=1\),
\[
    \sum_{j=1}^N
    \|\gamma(t_j)-\gamma(t_{j-1})\|
    =
    \lim_{m\to\infty}
    \sum_{j=1}^N
    \|\gamma_m(t_j)-\gamma_m(t_{j-1})\|
    \leq M.
\]
Taking the supremum over partitions proves the assertion.
\end{proof}

The recursive definition of the ordered tensor integrals gives the two
variation estimates
\begin{equation}
  \|I_q(\eta)_{s,t}\|
  \leq
  \frac{\Var(\eta;[s,t])^q}{q!},
  \qquad
  \Var_u\bigl(I_q(\eta)_{s,u};[s,t]\bigr)
  \leq
  \frac{\Var(\eta;[s,t])^q}{q!}.
  \label{eq:ordered-tensor-variation-estimates}
\end{equation}
For completeness, these follow by induction on \(q\).  Put
\(v(u)=\Var(\eta;[s,u])\).  The case \(q=1\) follows directly from
\(I_1(\eta)_{s,u}=\eta(u)-\eta(s)\).  If the estimates hold at level
\(q-1\), then \eqref{eq:ordered-tensor-recursion} and the
Riemann--Stieltjes variation estimate give
\[
  \|I_q(\eta)_{s,t}\|,
  \ \Var_u\bigl(I_q(\eta)_{s,u};[s,t]\bigr)
  \leq
  \int_s^t\frac{v(u)^{q-1}}{(q-1)!}\,dv(u)
  =
  \frac{v(t)^q}{q!}.
\]

\begin{lemma}[Continuity of the ordered integrals]
\label{lem:iterated-continuity}
Under the hypotheses of Lemma~\ref{lem:limit-BV}, for every fixed
\(q\geq 1\),
\[
    \sup_{0\leq s\leq t\leq 1}
    \bigl\|
       I_q(\gamma_m)_{s,t}-I_q(\gamma)_{s,t}
    \bigr\|
    \longrightarrow 0.
\]
\end{lemma}

\begin{proof}
We use induction on \(q\).

For \(q=1\),
\[
    I_1(\gamma_m)_{s,t}-I_1(\gamma)_{s,t}
    =
    (\gamma_m-\gamma)(t)-(\gamma_m-\gamma)(s),
\]
and hence
\[
    \sup_{s\leq t}
    \|I_1(\gamma_m)_{s,t}-I_1(\gamma)_{s,t}\|
    \leq
    2\|\gamma_m-\gamma\|_\infty
    \longrightarrow0.
\]

Assume the result holds at level \(q-1\), and put
\(Z_m=\gamma_m-\gamma\).  The recursive identity
\eqref{eq:ordered-tensor-recursion} gives
\begin{align}
    I_q(\gamma_m)_{s,t}-I_q(\gamma)_{s,t}
    ={}&
    \int_s^t
       d\gamma_m(u)\otimes
       \bigl(
          I_{q-1}(\gamma_m)_{s,u}
          -
          I_{q-1}(\gamma)_{s,u}
       \bigr)
       \notag\\
       &+
    \int_s^t
       dZ_m(u)\otimes I_{q-1}(\gamma)_{s,u}.
    \label{eq:iterated-integral-difference}
\end{align}
The first term is bounded by
\[
    M
    \sup_{s\leq u\leq t}
    \bigl\|
       I_{q-1}(\gamma_m)_{s,u}
       -
       I_{q-1}(\gamma)_{s,u}
    \bigr\|,
\]
which tends uniformly to zero by the induction hypothesis.

For the second term, put
\[
    Y_s(u)=I_{q-1}(\gamma)_{s,u}.
\]
Riemann--Stieltjes integration-by-parts gives
\begin{equation}
    \int_s^t dZ_m(u)\otimes Y_s(u)
    =
    Z_m(t)\otimes Y_s(t)
    -
    \int_s^t Z_m(u)\otimes dY_s(u),
    \label{eq:iterated-integral-integration-by-parts}
\end{equation}
because \(Y_s(s)=0\).  Therefore,
\begin{equation}
    \left\|
       \int_s^t dZ_m\otimes Y_s
    \right\|
    \leq
    \|Z_m\|_\infty
    \left(
       \|Y_s(t)\|
       +
       \Var(Y_s;[s,t])
    \right).
    \label{eq:iterated-integral-ibp-bound}
\end{equation}
By \eqref{eq:ordered-tensor-variation-estimates} and
Lemma~\ref{lem:limit-BV},
\begin{equation}
  \|Y_s(t)\|+\Var(Y_s;[s,t])
  \leq
  \frac{2\Var(\gamma;[s,t])^{q-1}}{(q-1)!}
  \leq
  \frac{2M^{q-1}}{(q-1)!}.
  \label{eq:iterated-integral-uniform-parenthesis-bound}
\end{equation}
This is uniform in \(s,t\).  Since
\[
    \|Z_m\|_\infty
    =
    \|\gamma_m-\gamma\|_\infty
    \longrightarrow0,
\]
the second term in \eqref{eq:iterated-integral-difference} also tends
uniformly to zero.  This completes the induction.
\end{proof}

\begin{proof}[Proof of Proposition~\ref{prop:pathwise-CF-continuity}]
Lemma~\ref{lem:limit-BV} gives that \(\gamma\) is a continuous BV path.  Lemma~\ref{lem:iterated-continuity} gives uniform convergence of every ordered tensor integral \(I_q(\gamma_j)_{0,u}\) to \(I_q(\gamma)_{0,u}\), uniformly in \(u\).  Applying the continuous left-nested bracket map gives uniform convergence of each \(T_q\) along the restricted paths \([0,u]\).

Proceed by induction in homogeneous depth.  The first component is the endpoint increment and converges uniformly.  At each later depth, the Chacon--Fomenko recursion expresses \(H_{k+1}\) as a finite linear combination of the already controlled components, the ordered bracket integrals, and continuous nested Lie brackets.  The induction is finite because the Lie algebra is nilpotent.  Thus, every homogeneous logarithmic component converges uniformly in \(u\).  Continuity of the exponential map on the compact set containing these uniformly bounded logarithms gives uniform convergence of the developed paths.
\end{proof}

\end{document}